\documentclass[12pt]{article}

\usepackage[english]{babel}

\usepackage[a4paper,top=2cm,bottom=2cm,left=2cm,right=2cm,marginparwidth=1.75cm]{geometry}

\usepackage{amsmath}
\usepackage{amsthm} 
\usepackage{dsfont}
\usepackage{graphicx}
\usepackage{mathrsfs}
\usepackage[colorlinks=true, allcolors=blue]{hyperref}
\usepackage{enumerate}

\def\n{{\boldsymbol n }}
\def\x{{\boldsymbol x }}
\def\a{{\boldsymbol a }}

\newcommand{\dx}{\,\mathrm{d}\boldsymbol{x}}

\newtheorem{theorem}{Theorem}
\newtheorem{lemma}[theorem]{Lemma}
\newtheorem{proposition}[theorem]{Proposition}
\newtheorem{corollary}[theorem]{Corollary}
\newtheorem{remark}[theorem]{Remark}
\newtheorem{definition}[theorem]{Definition}

\title{Finite element approximation and very weak solution existence in a two-dimensional, degenerate Keller-Segel model}
\author{Juan Vicente Gutiérrez-Santacreu\thanks{Dpto. de Matemática Aplicada I, E. T. S. I. Informática, Universidad de Sevilla. Avda. Reina Mercedes, s/n. E-41012 Sevilla, Spain.  E-mail: {\tt juanvi@us.es}. JVGS was partially supported by Grant I+D+I PID2023-149182NB-I00 funded by 
MICIU/AEI/10.13039/501100011033 and, ERDF/EU.
}}

\begin{document}
\maketitle



\begin{abstract}

This paper is devoted to the design and analysis of a numerical algorithm for approximating solutions of a degenerate cross-diffusion system, which models particular instances of taxis-type migration processes under local sensing mechanisms. The degeneracy leads to solutions that are very weak due to the low regularity themselves. Specifically, the solutions satisfy pointwise bounds (such as positivity and the maximum principle), integrability (such as mass conservation), and dual a priori estimates.

The proposed numerical scheme combines a finite element spatial discretization with Euler time stepping. The discrete solutions preserve the above-mentioned properties at the discrete level, enabling the derivation of compactness arguments and the convergence (up to a subsequence) of the numerical solutions to a very weak solution of the continuous problem on  two-dimensional polygonal domains.
\end{abstract}

\section{Introduction}
We consider a particular formulation of the Keller--Segel system modeling taxis-driven migration under local sensing mechanisms \cite{Keller_Segel_1971} contrasting with the commonly adopted gradient sensing framework.  The chemical signal is governed by a degenerate parabolic equation, where the nonlinear diffusion is modeled by a motility function $\Phi$, which stands for the sensitivity of the organisms to environmental cues. The organism density evolves according to a parabolic equation as well, where the organisms consume the chemical signal rather than producing it. 

Let $\Omega \subset \mathds{R}^2$ be a bounded domain with Lipschitz boundary $\partial \Omega$ and outward unit normal vector $\n$. For a fixed positive time $T > 0$, we define the spatiotemporal domain as $\Omega_T := \Omega \times (0, T]$ and $\partial\Omega_T=\partial\Omega\times(0,T]$. The system under consideration is governed by the following coupled partial differential equations:
\begin{equation}\label{model:KS}
\left\{
\begin{array}{rccl}
\partial_t u - \Delta (\Phi(v) u) &= 0 & \text{in} & \Omega_T, \\
\partial_t v - \Delta v + uv &= 0 & \text{in} & \Omega_T, \\
\partial_\n (\Phi(v) u) = \partial_\n v &= 0 & \text{on} & \partial \Omega_T, \\
u(0) = u_0, \quad v(0) &= v_0 & \text{in} & \Omega,
\end{array}
\right.
\end{equation}
where  $u(\x,t)$ and $v(\x,t)$ represent the cell density and the chemical concentration at spatial position $\x \in \bar\Omega$ and time $t \in [0, T]$, respectively. The motility function $\Phi: [0, \infty) \to (0, \infty)$ is a continuous function that models the interaction between the cell population and the chemical concentration, with the degenerate behavior at the origin.

The mathematical properties of model \eqref{model:KS} are closely linked to the conditions imposed on the motility function $\Phi$. For instance, mobility functions \cite{Tao_Winkler_2017} satisfying $ \Phi \in C^3([0, \infty))$, together with structural conditions such as
$$
k_\Phi\le  \Phi(s) \leq K_\Phi \quad \text{and} \quad |\Phi'(s)| \leq K_{\Phi'} \quad \text{for all } s > 0,
$$
where $ k_\Phi $, $ K_\Phi $, and $K_{\Phi'} $ are positive constants and under the additional assumption that $\Omega$ be convex, global-in-time classical solutions with uniform boundedness are known to exist in the two-dimensional case. In dimension three, the existence of solutions has so far been established only in the weak sense, but if $\|u_0\|_{L^2(\Omega)}+\|v_0\|_{W^{1,4}(\Omega)}$ is small enough such a weak solution is classical. Further, if only  $\|\Phi'\|_{L^\infty(0,\infty)} \|u_0\|_{L^1(\Omega)}$ is suitably small, the weak solution becomes eventually classical and bounded. With $\Phi\in C^3([0,\infty))$ such that $\Phi(s)>0$ on $[0,\infty)$, classical solutions are as well guaranteed in two dimensions, whereas three-dimensional solutions are merely weak \cite{Li_Winkler_2024}. Classical solutions globally in time have been as well found admitting degeneracy; more precisely,  $\phi\in C^1([0,\infty))\cap C^3((0,\infty))$ is such that $\phi(0)=0$, $\phi'(0)>0$,  and $\Phi(s)>0$ for all $s>0$ \cite{Winkler_2024}.      

It is reasonably to believe that if the regularity of $\Phi$ is weakened, as $\Phi \in C^0([0,\infty))$ with  $\Phi(s)>0$ for all $s\ge0$,  the concept of classical solvability is not reached leading to global existence of certain very weak solutions for any initial data $(u_0, v_0)\in(C^0(\bar\Omega))^*\times L^\infty(\Omega)$ likely lacking strong regularity properties \cite{Li_Winkler_2023}.  For motility functions $\Phi$ that are continuous and degenerate in the sense of \cite{Winkler_2023}, i.e, $\Phi\in C^0([0,\infty))$ with $\Phi(s)>0$ for all $s>0$ and such that there is a constant $\alpha>0$ satisfying 
$$
\liminf_{s\to 0} \frac{\Phi(s)}{s^\alpha}>0, 
$$
very weak solutions exist in this setting.

In this work, we are interested in this last framework, where we propose a numerical algorithm relied on a first-order finite element method for space, while time is discretized by an implicit Euler approximation; see section \ref{sec:FEM}.  We present an existence result on polygonal domains proving that the sequence of discrete solutions convergence towards a very weak solution.  The main challenges in analyzing numerically \eqref{model:KS} arise from the degeneracy of the nonlinear diffusion term $\Phi(v)$, which requires careful treatment in the numerical analysis; particularly, compactness and passage to the limit in sections \ref{sec:compactness} and \ref{sec:limit}, respectively. Lower and upper bounds associated with positivity, nonnegativity and the maximum principle are derived in section 	\ref{sec:nodal_bnds} for the discrete solutions thanks to the use of weakly acute meshes. Moreover, as customary, when testing against functions that do not belong to the test space, mass lumping is applied in \eqref{Alg} to the time terms to face a priori bounds in section \ref{sec:a_priori_bnds}.    

Extensions to other numerical techniques \cite{Huang_Shen_2021, Acosta-Soba_Guillen-Gonzalez_Rodriguez-Galvan_2023, Badia_Bonilla_GS_2023}, developed for the classical Keller--Segel equations, which are known to produce numerical solutions maintaining all the essential properties from the continuous analysis, are not readily adaptable to the numerical analysis presented in this work. 

\section{The finite element approximation}\label{sec:FEM}
\subsection{Space setting}
We assume that $\Omega \subset \mathds{R}^2$ is a polygonal domain, and we introduce a quasi-uniform family $\{\Sigma_h\}_{h>0}$ of triangulations of $\bar\Omega$ into disjoint closed simplices, i.e, $\Omega = \bigcup_{\sigma \in \Sigma^h} \sigma$, with local mesh size $h_\sigma := \text{diam}(\sigma)$ and global mesh size $h := \max_{\sigma \in \Sigma_h} h_\sigma$.  In addition, it is assumed that $\Sigma_h$ is a (weakly) acute partitioning; that is, the sum of opposite angles relative to any side does not exceed $\pi$.

Let $\{\boldsymbol{a}_i\}_{i=0}^I$ be the coordinates of the nodes of $\Sigma_h$.  The finite element space is constructed using continuous, piecewise linear functions defined on $\Sigma_h$:
$$
X_h := \left\{ \chi \in C^0(\bar\Omega) : \chi|_\sigma \text{ is linear for all } \sigma \in \Sigma_h \right\} \subset H^1(\Omega),
$$
with standard basis functions $\{\varphi_{\a_i}\}_{i=1}^I$; that is, $\varphi_{\a_i} \in X_h$ and $\varphi_{\a_i}(\a_j) = \delta_{ij}$ for all $i, j = 0,\cdots, I$. Further, 
$$
X_h^{\int=0}=\{\chi_h\in X_h\,:\, \int_\Omega \chi_h \,{\rm d}\x=0\}.
$$
Some local relations between the Lebesgue and Sobolev norms on $X_h$ are set up: 
\begin{enumerate}[(a)]
\item There is a constant $C_{\rm inv}>0$, independent of $h$, such that, for all $p\in[1,\infty]$, 
\begin{equation}\label{inv_ineq_Wnp->Wmq}
\|\chi_h\|_{W^{n,p}(\sigma)}\le C_{\rm inv} h^{m-n+\min\{0, \frac{2}{p}-\frac{2}{q}\}} \|\chi_h\|_{W^{m,q}(\sigma)}\quad\mbox{for all}\quad \chi_h\in X_h.
\end{equation}

See \cite[lm 4.5.3]{Brenner_Scott_2008} and \cite[lm 1.138]{Ern_Guermond_2004} for a proof that does not require the quasi-uniformity of $\Sigma_h$. However, under this assumption, the constant $C_{\rm err}$ becomes independent of $h$.
\end{enumerate}

We consider $\mathcal{I}_h : C(\bar\Omega) \to X_h$, the nodal interpolation operator, such that $\mathcal{I}_h \chi(\a_i) = \chi(\a_i)$ for $i = 0,\cdots, I$. For each $i\in I$, let $I_{\a_i}=\{ i\in I\, :\, \a_i\in\textrm{supp }\varphi_{\a_i}\}$.  

The following properties of $\mathcal{I}_h$ will be used throughout this paper: 

\begin{enumerate}
\item[(b)] There is a constant $C_{\rm err}>0$, independent of $h$, such that, for all $\sigma \in \Sigma_h$ and  $\chi\in W^{2,\infty}(\sigma)$,  
\begin{equation}\label{err_Linf-W2inf:I_h}
\| \chi - \mathcal{I}_h[\chi] \|_{L^\infty(\sigma)} \le C_{\rm err } h  \|\nabla^2\chi\|_{L^\infty(\sigma)}
\end{equation}
and 
\begin{equation}\label{err_L1-W21:I_h}
\|\chi-\mathcal{I}_h[\chi]\|_{L^1(\sigma)}\le C_{\rm err} h^2\, \|\nabla^2\chi\|_{L^1(\sigma)},
\end{equation}
See \cite[Th. 4.4.4]{Brenner_Scott_2008} and \cite[Th. 1.103]{Ern_Guermond_2004} for a proof that as before is not relied on the quasi-uniformity of $\Sigma_h$. With this, $C_{\rm err}$ is independent of $h$.   
\end{enumerate}
A discrete semi-inner product on $C(\bar\Omega)$ is defined by
$$
(\chi, \bar\chi)_{h} :=\int_\Omega \mathcal{I}_h[\bar \chi_h \chi] \,{\rm d}\x = \sum_{i\in I} \chi_h(\a_i) \bar\chi_h(\a_i) \|\varphi_{\a_i}\|_{L^1(\Omega)}.
$$
We also introduce the norm $ \|\chi\|_h := \left[ (\chi, \chi)_h \right]^{1/2}$  for all $\chi \in C^0(\bar\Omega)$, which satisfies the following equivalence with $\|\cdot\|_{L^2(\Omega)}$. 
\begin{enumerate}
\item[(c)] There is a constant $C_{\rm eq}>0$, independent of $h$, such that, for all $\chi_h\in X_h$,  
\begin{equation}\label{equiv:L2_and_L2h}
\|\chi_h\|_{L^2(\Omega)}\le \|\chi_h\|_h\le C_{\rm eq}\|\chi_h\|_{L^2(\Omega)}.
\end{equation}
See \cite[Prop. 2.3]{GG_GS_2019} for a proof. 
\end{enumerate}
Some properties concerning commutator errors for $\mathcal{I}_h$ will be needed as well.
\begin{enumerate}
\item[(d)] There is a constant $C_{\rm com}>0$, independent of $h$, such that 
\begin{equation}\label{com_err_L1->L2-H1:I_h}
|(\chi_h, \bar \chi_h)_h-(\chi_h, \bar\chi_h)|\le C_{\rm com} h  \|\chi_h\|_{L^2(\Omega)} \|\nabla\bar \chi\|_{L^2(\Omega)}.
\end{equation}
and
\begin{equation}\label{com_err_L1->L1-W1inf:I_h}
\|\chi_h \bar \chi_h-\mathcal{I}_h[\chi_h\bar\chi_h]\|_{L^1(\sigma)}\le C_{\rm err} h\|\chi_h\|_{L^1(\sigma)}  \|\nabla\bar\chi_h\|_{L^\infty(\sigma)}. 
\end{equation}

See \cite{GG_GS_2019} for a proof of \eqref{com_err_L1->L2-H1:I_h}. For \eqref{com_err_L1->L1-W1inf:I_h}, observe that 
$$
\nabla^2(\chi_h \bar\chi_h)=2\sum_{i,j=1}^2 \partial_i \chi_h \partial_j \bar\chi_h.
$$
Using \eqref{err_L1-W21:I_h} and \eqref{inv_ineq_Wnp->Wmq} yields
$$
\begin{array}{rcl}
\|\mathcal{I}_h(\chi_h\bar\chi_h)-\chi_h \bar\chi_h\|_{L^1(\sigma)}&\le&\displaystyle 2 C_{\rm err} h^2 \int_\sigma |\nabla\chi_h| |\nabla\bar\chi_h |
\\
&\le&\displaystyle 
2 C_{\rm app} h^2 \|\nabla\chi\|_{L^1(\sigma)} \|\nabla\bar\chi_h\|_{L^\infty(\sigma)} 
\\
&\le&\displaystyle
2 C_{\rm app} C_{\rm inv} h \|\chi\|_{L^1(\sigma)} \|\nabla\bar\chi_h\|_{L^\infty(\sigma)}.
\end{array}
$$
\end{enumerate}

\subsection{Algorithm}

Let $(u_0, v_0)\in L^1(\Omega)\cap (H^1(\Omega))'\times L^\infty(\Omega)$ with $u_0\ge0$ and $v_0>0$ a. e. on $\Omega$. We consider $(u_{0h}, v_{0h})\in X_h^2$ to be approximations of $(u_0, v_0)$ constructed through an averaged interpolation operator $ \mathcal{Q}_h: L^1(\Omega)\to X_h$ defined by 
\begin{equation}\label{def:Q_h}
(\mathcal{Q}_h \chi, \chi_h)_h=(\chi, \chi_h)\quad\mbox{for all}\quad \chi_h\in X_h. 
\end{equation}

\begin{proposition} The following properties of $\mathcal{Q}_h$ are satisfied:
\begin{equation}\label{post-nonneg:Q_h}
\mathcal{Q}_h\chi >0\quad\mbox{or}\quad\ge0\quad\mbox{ if }\quad \chi >0\quad\mbox{or}\quad\ge0.
\end{equation}
\begin{equation}\label{sta_Lp:Q_h}
\|\mathcal{Q}_h\chi\|_{L^p(\Omega)}\le \|\chi\|_{L^p(\Omega)}\quad\mbox{ for } p=1, 2, \infty,
\end{equation}
and
\begin{equation}\label{sta_H1':Q_h}
\|\mathcal{Q}_h\chi\|_{(H^1(\Omega))'}\le C_{\rm sta} \|\chi\|_{(H^1(\Omega))'}.
\end{equation}

\end{proposition}
\begin{proof} Assertion \eqref{post-nonneg:Q_h} is a direct consequence of the definition. For assertion \eqref{sta_Lp:Q_h}, note that 
$$
\mathcal{Q}_h\chi=\sum_{i\in I}\frac{(\chi,\varphi_{\a_j})}{\|\varphi_{\a_i}\|_{L^1(\Omega)}} \varphi_{\a_i}.
$$
This implies $\|\mathcal{Q}_h\chi\|_{L^\infty(\Omega)}\le\|\chi\|_{L^\infty(\Omega)}$, since
$$
\frac{(\chi,\varphi_{\a_j})}{\|\varphi_{\a_i}\|_{L^1(\Omega)}}\le \|\chi\|_{L^\infty(\Omega)}.
$$
Additionally,   
$$
\|\mathcal{Q}_h\chi\|_{L^1(\Omega)}\le \sum_{i\in I} \frac{|(\chi, \varphi_{\a_i})|}{\|\varphi_{\a_i}\|_{L^1(\Omega)}} \|\varphi_{\a_i}\|_{L^1(\Omega)} = \sum_{i\in I} \int_\Omega |\chi| \varphi_{\a_i} {\rm d}\x=\|\chi\|_{L^1(\Omega)}. 
$$
The stability for $p=2$ is trivial. 

Assertion \eqref{sta_H1':Q_h} will now be demonstrated. In doing so, we need to prove that there is a constant $C_{\rm sta}>0$, independent of $h$, such that  
\begin{equation}\label{stab_Qh:dual_H1}
\|\mathcal{Q}_h\chi\|_{H^1(\Omega)}\le C_{\rm sta} \|\chi\|_{H^1(\Omega)}.
\end{equation}

Let $\mathcal{P}_h\chi\in X_h$ solve
\begin{equation}\label{def:P_h}
(\mathcal{P}_h \chi, \chi_h)_h=(\chi, \chi_h)\quad\mbox{for all}\quad \chi_h\in X_h, 
\end{equation}
which satisfies \cite[Lm 1.131]{Ern_Guermond_2004}, for all $\chi\in H^1(\Omega)$,  
\begin{equation}\label{sta_H1:P_h}
\|\mathcal{P}_h\chi\|_{H^1(\Omega)}\le C_{\rm sta} \|\chi\|_{H^1(\Omega)}.
\end{equation}
Next, we see, from \eqref{def:Q_h} and \eqref{def:P_h}, that   
$$
(\mathcal{P}_h\chi- \mathcal{Q}_h\chi, \chi_h)_h=(\mathcal{P}_h\chi, \chi_h)_h-(\mathcal{P}_h\chi, \chi_h).
$$
Taking $\chi_h= \mathcal{P}_h\chi - \mathcal{Q}_h\chi $ yields, on noting \eqref{com_err_L1->L2-H1:I_h}, \eqref{equiv:L2_and_L2h}, and \eqref{sta_H1:P_h},  that 
\begin{equation}\label{pro1:lab1}
\|\mathcal{P}_h\chi - \mathcal{Q}_h\chi\|_h \le C_{\rm com}C_{\rm sta} h \|\chi\|_{H^1(\Omega)} 
\end{equation}

In order to deal with the stability of $\mathcal{Q}_h$ in $H^1(\Omega)$, we use \eqref{pro1:lab1}, \eqref{inv_ineq_Wnp->Wmq}, and \eqref{sta_H1:P_h} to obtain      
$$
\begin{array}{rcl}
\|\nabla\mathcal{Q}_h\chi\|_{L^2(\Omega)}&\le& \|\nabla\mathcal{Q}_h\chi - \nabla\mathcal{P}_h\chi\|_{L^2(\Omega)} + \|\nabla\mathcal{P}_h\chi\|_{L^2(\Omega)}
\\
&\le& C_{\rm inv} h^{-1} \|\mathcal{Q}_h\chi - \mathcal{P}_h\chi\|_{L^2(\Omega)} + C_{\rm sta} \|\mathcal{P}_h\chi\|_{H^1(\Omega)}
\\
&\le& C_{\rm inv} C_{\rm com} C_{\rm sta} \|\chi\|  + C_{\rm sta} \|\chi\|_{H^1(\Omega)},  
\end{array}
$$
which combined with inequality \eqref{sta_Lp:Q_h} for $p=2$ leads to \eqref{sta_H1:P_h}. Now we focus on seeing that \eqref{sta_H1':Q_h} is true. To carry this out, we have, by \eqref{def:Q_h}  and \eqref{stab_Qh:dual_H1}, that
$$
\begin{array}{rcl}
\|\mathcal{Q}_h\chi\|_{(H^1(\Omega))'}&=&\displaystyle
\sup\frac{(\mathcal{Q}_h\chi,\bar\chi)}{\|\bar\chi\|_{H^1(\Omega)}}
\\
&\le&\displaystyle
\sup\frac{(\chi,\mathcal{Q}_h\bar\chi)}{\|\bar\chi\|_{H^1(\Omega)}}
\\
&\le& C_{\rm sta} \|\chi\|_{(H^1(\Omega))'}.
\end{array}
$$

\end{proof}

This construction of $\mathcal{Q}_h$ ensures the following properties:
\begin{equation}\label{prop:u_0h}
\left\{
\begin{array}{rcl}
\|u_{0h}\|_{L^1(\Omega)} &\le& \|u_0\|_{L^1(\Omega)}, \\[1ex]
\|u_{0h}\|_{(H^1(\Omega))'} &\le& C_{\mathrm{sta}} \|u_0\|_{(H^1(\Omega))'}, \\[1ex]
u_{0h}(\boldsymbol{x}) &\ge& 0 \quad \text{for all } \boldsymbol{x} \in \Omega,
\end{array}
\right.
\end{equation}
and, similarly,
\begin{equation}\label{prop:v_0h}
\|v_{0h}\|_{L^\infty(\Omega)} \le \|v_0\|_{L^\infty(\Omega)} \quad \text{and} \quad v_{0h}(\boldsymbol{x}) > 0 \quad \text{for all } \boldsymbol{x} \in \Omega.
\end{equation}
 
For $T>0$, the time discretization is performed by dividing the interval $[0,T]$ into uniform subintervals of size $k=T/N$ with $N\in\mathds{N}$. This gives a time sequence $\{\mathcal{T}_k\}_{k>0}$, where $\mathcal{T}_k=\bigcup_{\tau\in\mathcal{T}_h}\tau $ with $\tau=[t_n, t_{n+1}]$ for some $n\in\{0,\cdots, N-1\}$. The discrete algorithm is then formulated 
as: given $(u_{0h}, v_{0h})\in X_h^2$,  we want to know $(u^{n+1}_h, v^{n+1}_h)\in X_h^2$ such that, for all $(\bar u_h, \bar v_h)\in X_h^2$,  
\begin{equation}\label{Alg}
\left\{
\begin{array}{rcl}
(\delta_t u^{n+1}_h, \bar u_h)_h+(\nabla\mathcal{I}_h(\Phi(v^n_h) u^{n+1}_h), \nabla\bar u_h)&=&0,
\\
(\partial_t v^{n+1}_h, \bar v_h)_h+(\nabla v^{n+1}_h, \nabla\bar v_h)+(v^{n+1}_h u^{n+1}_h,\bar v_h)_h&=&0.
\end{array}
\right.
\end{equation}
It results in a decoupled system of linear equations for the discrete unknowns $u_h^{n+1}$ and $v_h^{n+1}$ at each time step. Therefore, existence follows from uniqueness. Indeed, suppose that there are two solution pairs $(u^{n+1}_h, v^{n+1}_h)$ and $(\tilde u^{n+1}_h, \tilde v^{n+1}_h)$ and compare both to find, on defining $w^{n+1}_h=u^{n+1}_h-\tilde u^{n+1}_h$, that  
$$
(\delta_t w^{n+1}_h, \bar u_h)_h+(\nabla\mathcal{I}_h(\Phi(v^n_h) w^{n+1}_h), \nabla\bar u_h)=0.
$$
Making use of the development to deduce \eqref{lm6:lab2} follows
$$
\|\psi^{n+1}_h\|^2_{H^1_h(\Omega)}+ 2 k(\Phi(v^n_h)w^{n+1}_h, w^{n+1}_h)_h\le 2 k \max_{x\in [0,\|v^n_h\|_{L^\infty(\Omega)}]} \Phi(x) \|\psi^{n+1}_h\|^2_{H^1_h(\Omega)}.
$$
Upon election of $2k \max_{x\in [0,\|v^{0h}\|_{L^\infty(\Omega)}]} \Phi(x)< 1$, we deduce that $\psi^{n+1}_h\equiv0$, which implies, in turn, that $u^{n+1}_h\equiv0$ from  \eqref{eq:psi_h}. 
\subsection{Main result}
We now state the concept of solvability for system \eqref{Alg}, as introduced in \cite{Winkler_2023}.  
\begin{definition}\label{def:very_weak_sol}
Let $T>0$ be fixed. For given $u_0 \in L^1(\Omega)$ and $v_0 \in L^\infty(\Omega)$ with $u_0\ge 0$ and $v_0>0$ a. e. on $\Omega$, a pair $(u, v)$ fulfiling 
\begin{align*}
u &\in L^\infty(0,T; L^1(\Omega))\quad \mbox{with}\quad u\ge0\mbox{ a. e. on }\Omega\times(0,T], \\
v &\in L^\infty(\Omega_T ) \cap L^2(0, T; H^1(\Omega))\quad \mbox{with}\quad v>0\mbox{ a. e. on }\Omega\times(0,T]
\end{align*}
is a global very weak solution of \eqref{model:KS} if 
\begin{equation} \label{def:weak_sol_u}
\int_0^T \int_\Omega -u \varphi_t \, \,{\rm d}\x\, {\rm d}t - \int_\Omega u_0 \varphi(\cdot, 0) \, \,{\rm d}\x = \int_0^T \int_\Omega \Phi(v) u  \Delta \varphi \, \,{\rm d}\x\, {\rm d}t,
\end{equation}
for each $\varphi \in W^{1,1}(0,T; H^2(\Omega))$ with $\partial_\n \varphi = 0$ on $\partial\Omega \times (0, T]$ and $\varphi(T)=0$, and if, moreover,
\begin{equation} \label{def:weak_sol_v}
\int_0^T \int_\Omega v \varphi_t \, \,{\rm d}\x\, {\rm d}t + \int_\Omega v_0 \varphi(\cdot, 0) \, \,{\rm d}\x = \int_0^T \int_\Omega \nabla v \cdot \nabla \varphi \, \,{\rm d}\x\, {\rm d}t + \int_0^T \int_\Omega uv \varphi \, \,{\rm d}\x\, {\rm d}t,
\end{equation}
for each $\varphi \in W^{1,1}([0,T); H^1(\Omega))$ with $\varphi(T)=0$.
\end{definition}

To facilitate the analysis and the passage to the limit as $h, k \to 0$, we introduce some auxiliary definitions. Given a sequence $\{w_h^n\}_{n=0}^N \subset X_h$, we define the following piecewise and continuous functions:
\begin{equation}\label{def:w_hk_tilde}
\widetilde w_{h,k}: [0,T]\to X_h\quad\mbox{ such that }  \widetilde w_{h,k}(t) := w_h^{n},\quad  t \in [t_n, t_{n+1}), 
\end{equation}
\begin{equation}\label{def:w_hk_hat}
\widehat w_{h,k}: [0,T]\to X_h\quad\mbox{ such that } \widehat w_{h,k}(t) := w_h^{n+1},\quad t \in [t_n, t_{n+1}), 
\end{equation}
and
\begin{equation} \label{def:w_hk}
\begin{array}{c}
w_{h,k}: [0,T]\to X_h\quad\mbox{ such that }
\\
\displaystyle
w_{h,k}(t) := w_h^{n+1} + \frac{w_h^{n+1} - w_h^n}{k}(t - t_{n+1}), \quad t \in [t_n, t_{n+1}).
\end{array}
\end{equation}

\begin{theorem}\label{th:main} Let $\Phi\in C([0,\infty)$ such that $\Phi(x)>0$ for all $x\in[0,\infty)$ and 
\begin{equation}\label{cond:Phi}
\liminf_{s\to 0} \frac{\Phi(s)}{s^\alpha}>0.
\end{equation}
Then there exists a subsequence of $\{\widetilde u_{h,k}, \widetilde v_{h,k} \}_{h,k>0}$, $\{ \widehat u_{h,k}, \widehat v_{h,k} \}_{h,k>0}$, and $\{ u_{h,k}, v_{h,k} \}_{h,k>0}$ that converges to limit functions $(u,v)$ satisfying the requirements in Definition  \eqref{def:very_weak_sol}.   
\end{theorem}

The proof of Theorem \ref{th:main} is divided into the following four sections. 

\section{Nodal bounds}\label{sec:nodal_bnds}
The acuteness condition ensures key properties of the resulting finite element matrices through \eqref{Alg}, particularly non-negativity, positivity, and the discrete maximum principle for the discrete solution sequence pair $\{(u^n_h, v^n_h)\}_{n=0}^N$. 

\begin{lemma} For $n=0, \cdots, N$, there hold, for all  $\boldsymbol{x}\in\Omega$, that: 
\begin{itemize}
\item Nonnegativity:
\begin{equation}\label{posit:u^n_h}
u^n_h(\boldsymbol{x})\ge 0.
\end{equation}
\item Positivity and maximum principle:
\begin{equation}\label{Positivity_and_DMP:v^n_h}
0<v_h^n(\boldsymbol{x})\le \|v_{0h}\|_{L^\infty(\Omega)}.
\end{equation} 
\end{itemize}
\end{lemma}
\begin{proof}

Assertions \eqref{posit:u^n_h} and \eqref{Positivity_and_DMP:v^n_h}  hold for $n = 0$ in view of our assumptions on $(u_{0h}, v_{0h})$ in \eqref{prop:u_0h} and \eqref{prop:v_0h}, respectively. Suppose now that they are true for some $n \in \{0, \dots, N - 1\}$ and false for $n+1$.  Then,  we proceed by contradiction.

$\bullet$ Nonnegativity in \eqref{posit:u^n_h}. Given $\chi_h\in X_h$, define $\chi_h^+$ and $\chi_h^-$ be the positive and negative parts of $\chi_h$, respectively. Substituting $u_h=\mathcal{I}_h[\Phi(v^n_h) u^{n+1,-}_h] $ into $\eqref{Alg}_1$, we obtain
$$
(\delta_t u^{n+1}_h, \Phi(v^n_h) u^{n+1,+}_h )_h +(\nabla\mathcal{I}_h[\Phi(v^n_h) u^{n+1}_h], \nabla\mathcal{I}_h[\phi(v^n_h)u^{n+1,-}_h])=0.
$$
Observe that 
$$
\begin{array}{rcl}
(\delta_t u^{n+1}_h, \Phi(v^n_h) u^{n+1,-}_h )_h &=&\displaystyle
\frac{1}{k} \sum_{i\in I}  \Phi(v^n_h(\a_i))\Big[(u^{n+1,-}_h(\a_i))^2 
- u^n_h(\a_i) u^{n+1,-}_h (\a_i) \Big]  \|\varphi_{\a_i}\|_{L^1(\Omega)} 
\\
&\ge&\displaystyle
\sum_{i\in I} \Phi(v^n_h(\a_i))(u^{n+1,-}_h(\a_i))^2 \|\varphi_{\a_i}\|_{L^1(\Omega)}
\\
&=&\|\sqrt{\Phi(v^n_h)} u^{n+1,-}_h\|^2_h,
\end{array}
$$
since $u^n_h(\a_i) u^{n+1,-}_h (\a_i)<0$, and  
$$
\begin{array}{rcl}
(\nabla\mathcal{I}_h[\Phi(v^n_h) u^{n+1}_h], \nabla\mathcal{I}_h[\phi(v^n_h)u^{n+1,+}_h])&=&\|\nabla\mathcal{I}_h[\Phi(v^n_h) u^{n+1,-}_h]\|_{L^2(\Omega)}^2
\\
&+&(\nabla\mathcal{I}_h[\Phi(v^n_h) u^{n+1,+}_h], \nabla\mathcal{I}_h[\Phi(v^n_h) u^{n+1,-}_h])
\\
&=&\|\nabla\mathcal{I}_h[\Phi(v^n_h) u^{n+1,-}_h]\|_{L^2(\Omega)}^2. 
\end{array}
$$
Therefore, 
$$
\|\sqrt{\Phi(v^n_h)}u^{n+1,-}_h\|_h^2+\|\nabla\mathcal{I}_h[\Phi(v^n_h) u^{n+1,-}_h]\|_{L^2(\Omega)}^2\le 0. 
$$
From the fact that $\Phi(v^n_h)>0$, since $v^n_h>0$ from our induction hypothesis, there holds \break $\|\sqrt{\Phi(v^n_h)}u^{n+1,-}_h)\|_h^2=0$. This observation leads to $u^{n+1,-}_h(\a_i)=0$ for $i\in I$, and gives a contradiction. 

$\bullet$ Positivity in \eqref{Positivity_and_DMP:v^n_h}. Fixing $i \in \{1, \dots, I\}$, for which we know that there is at least a negative minimum of $v^{n+1}_h$ at $\x=\a_i\in\mathcal{N}_h$, and inserting $\bar  u_h = \varphi_{\a_i}$ into $\eqref{Alg}_2$, we have
$$
\begin{array}{c}
\displaystyle
 v^{n+1}_h(\a_i) \|\varphi_{\a_i}\|_{L^1(\Omega)} + k \sum_{j\in I} v^{n+1}_h(\a_j) (\nabla \varphi_{\a_j}, \nabla\varphi_{\a_i})
\\ 
+u^{n+1}_h(\a_i) v^{n+1}_h(\a_i) \|\varphi_{\a_i}\|_{L^1(\Omega)} = v_h^n(\a_j)\|\varphi_{\a_i}\|_{L^1(\Omega)}.
\end{array}
$$
As $(\nabla\varphi_{\a_j}, \nabla\varphi_{\a_i})\le 0 $ for $i\not=j$, and $v_h^{n+1}(\a_i)\le v_h^{n+1}(\a_j)$, for all $j\in I_{\a_i}$, we infer that
$$
(\nabla\varphi_{\a_j}, \nabla\varphi_{\a_i}) v_h^{n+1}(\a_j)\le (\nabla\varphi_{\a_j}, \nabla\varphi_{\a_i}) v_h^{n+1}(\a_i), 
$$
which implies, in turn, on noting $\sum_{j\in I}\nabla\varphi_{\a_j}=0$ and $u^{n+1}_h(\a_i)\ge0$, that  
$$
v^{n+1}_h(\a_i) \|\varphi_{\a_i}\|_{L^1(\Omega)}  \ge v_h^n(\a_i) \|\varphi_{\a_i}\|_{L^1(\Omega)} .
$$
As $v^n_h(\a_i)>0$ by induction hypothesis, this results in a contradiction, since  $u^{n+1}_h(\a_i)<0$. 

$\bullet$ Maximum principle for $v^{n+1}_h$.  Recall by induction that $v^n_h\le \|v_{0h}\|_{L^\infty(\Omega)}$ is true for $n\in\{0,\cdots, N-1\}$. Further we assume  that $v^{n+1}_h$ has a local maximum at $\x=\a_i\in\mathcal{N}_h$, for which $v^{n+1}_h(\a_i)>\|v_{0h}\|_{L^\infty(\Omega)}$ holds. Let now $\bar v_h=\varphi_{\a_i}$ in $\eqref{Alg}_2$ to get 
$$
\begin{array}{rcl}
\displaystyle
v^{n+1}_h(\a_i) \|\varphi_{\a_i}\|_{L^1(\Omega)}+k\sum_{j\in I_{\a_i}} v^{n+1}_h(\a_j)(\nabla\varphi_{\a_j} , \nabla \varphi_{\a_i})&&
\\
+u^{n+1}_h(\a_i) v^{n+1}_h(\a_i) \|\varphi_{\a_i}\|_{L^1(\Omega)}&=& v^n_h(\a_i) \|\varphi_{\a_i}\|_{L^1(\Omega)}. 
\end{array}
$$
Thus, conjugating $(\nabla\varphi_{\a_j}, \nabla\varphi_{\a_i})\le 0$ with $v^{n+1}_h (\a_j)\le v_h^{n+1}(\a_i)$, for all $j\in I_{\a_j}$, yields, on noting $u^{n+1}_h(\a_i)\ge0$, that 
$$
v^{n+1}_h(\a_i) \Big(\|\varphi_{\a_i}\|_{L^1(\Omega)}+u^{n+1}_h(\a_i) \|\varphi_{\a_i}\|_{L^1(\Omega)} \Big) \le \|\varphi_{\a_i}\|_{L^1(\Omega)} v^{n+1}_h(\a_i)
$$
and hence
$$
v^{n+1}_h(\a_i) \le v^n_h(\a_i),
$$
which contradicts our hypothesis.  
\end{proof}
\section{A priori estimates}\label{sec:a_priori_bnds}
 Our first step is to establish a bound for the total mass of the sequence pair $\{(u^n_h, v^n_h)\}_{n=0}^N$. 
\begin{lemma}
For all $n \in \{0,\cdots, N\}$, it follows that
\begin{equation}\label{mass_conv:u^n_h}
\|u^n_h\|_{L^1(\Omega)} = \|u_{0h}\|_{L^1(\Omega)},
\end{equation}
and
\begin{equation}\label{bnd_L1:v^n_h}
\|v^n_h\|_{L^1(\Omega)} \le \|v_{0h}\|_{L^1(\Omega)}.
\end{equation}
\end{lemma}

\begin{proof}
We first test equation~$\eqref{Alg}_1$ with $\bar u_h = 1$, which yields
$$
\int_\Omega u^{n+1}_h\,\dx = \int_\Omega u^n_h\,\dx,
$$
and thus \eqref{mass_conv:u^n_h} follows from~\eqref{posit:u^n_h}.

Next, we test equation~$\eqref{Alg}_2$ with $\bar v_h = 1$, obtaining
$$
\int_\Omega v^{n+1}_h\,\dx + (u^{n+1}_h, v^{n+1}_h)_h = \int_\Omega v^n_h\,\dx.
$$
From \eqref{posit:u^n_h} and \eqref{Positivity_and_DMP:v^n_h}, we deduce that \eqref{bnd_L1:v^n_h} holds.
\end{proof}
A first consequence of the control of the $L^1(\Omega)$-norm is the following dual bound.
\begin{lemma} Let $k>0$ such that 
\begin{equation}\label{smallness:k}
2 k\max_{x\in [0,\|v_{0h}\|_{L^\infty(\Omega)}]} \Phi(x) < 1.
\end{equation}
Then there is   a constant $C_1=C_1(u_0, v_0, \Phi, T) > 0$, independent of $(h,k)$, satisfying 
\begin{equation}\label{bnd:Phi(v^n_h)u^(n+1)_h in L2L2h}
\max_{n\in\{0, \cdots, N-1\}}\|\Psi^{n+1}_h\|^2_{H^1_h(\Omega)}+\displaystyle
k\sum_{n=0}^{N-1} (k^2 \|\delta_t\Psi^{n+1}_h\|_{H^1_h(\Omega)}^2+(\Phi(v^n_h)u^{n+1}_h, u^{n+1}_h)_h) \le C_1.
\end{equation}
with $\Psi^{n+1}_h$ defined in \eqref{eq:psi_h}. 
\end{lemma}
\begin{proof}
Let $\psi^{n+1}_h\in X_h$ be such that, for all $\bar\psi_h\in X_h$,  
\begin{equation}\label{eq:psi_h}
(\nabla\psi^{n+1}_h, \nabla\bar\psi_h)+(\psi^{n+1}_h, \bar\psi_h)_h=(u^{n+1}_h, \bar\psi_h)_h.
\end{equation}
As $u^{n+1}_h\ge0 $ from \eqref{posit:u^n_h}, one can prove that 
\begin{equation}\label{posit:Psi_h}
\Psi^{n+1}_h\ge0,
\end{equation} 
since the matrix resulting from \eqref{eq:psi_h} is an $M$-matrix.
  
From now on, we denote $\|\cdot\|_{H^1_h(\Omega)}=\|\nabla\cdot\|_{L^2(\Omega)}+\|\cdot\|_h$. Take $\bar u_h=\psi^{n+1}_h$ in $\eqref{Alg}_1$ to get 
\begin{equation}\label{lm6:lab1}
(\delta_t u^{n+1}_h, \psi^{n+1}_h)_h+(\nabla \mathcal{I}_h[\Phi(v^n_h) u^{n+1}_h], \nabla \psi^{n+1}_h)=0.
\end{equation}
It is not hard to see from $\eqref{eq:psi_h}$ that
\begin{equation}\label{eq:delta_tpsi_h}
(\nabla\delta_t\psi^{n+1}_h, \nabla\bar\psi_h)+(\delta_t\psi^{n+1}_h, \bar\psi_h)_h=(\delta_t u^{n+1}_h, \bar\psi_h)_h.
\end{equation}
Choosing $\bar\psi_h=\psi^{n+1}_h$ in \eqref{eq:delta_tpsi_h} and substituting this into \eqref{lm6:lab1} yields 
$$
\|\psi^{n+1}_h\|^2_{H^1_h(\Omega)}-\|\psi^n_h\|^2_{H^1_h(\Omega)}+\|\psi^{n+1}_h-\psi^n_h\|^2_{H^1_h(\Omega)}+2k(\nabla\mathcal{I}_h[\Phi(v^n_h) u^{n+1}_h], \nabla\psi^{n+1}_h)=0.
$$ 
Next select $\bar\psi_h=\mathcal{I}_h[\phi(v^n_h) u^{n+1}_h]$ in \eqref{eq:psi_h}  to have
$$
(\nabla\psi^{n+1}_h, \nabla\mathcal{I}_h[\Phi(v^n_h) u^{n+1}_h])+(\psi^{n+1}_h, \phi(v^n_h) u^{n+1}_h)_h=(u^{n+1}_h, \Phi(v^n_h) u^{n+1}_h)_h
$$ 
and hence 
$$
\|\psi^{n+1}_h\|^2_{H^1_h(\Omega)}-\|\psi^n_h\|^2_{H^1_h(\Omega)}+\|\psi^{n+1}_h-\psi^n_h\|^2_{H^1_h(\Omega)} + 2 k (\Phi(v^n_h)u^{n+1}_h, u^{n+1}_h)_h=2 k (\psi^{n+1}_h, \Phi(v^n_h) u^{n+1}_h)_h.
$$ 
In view of \eqref{posit:Psi_h} and \eqref{Positivity_and_DMP:v^n_h}, we bound:
$$
\begin{array}{rcl}
(\psi^{n+1}_h, \Phi(v^n_h) u^{n+1}_h)_h&=&\displaystyle
\sum_{i\in I} \psi^{n+1}_h(\a_i) \Phi(v^n_h(\a_i)) u^{n+1}_h(\a_i) \|\varphi_{\a_i}\|_{L^1(\Omega)}
\\
&\le&\displaystyle
\max_{s\in [0,\|v_{0h}\|_{L^\infty(\Omega)}]} \Phi(s)  \sum_{i\in I}  \psi^{n+1}_h(\a_i)  u^{n+1}_h(\a_i) \|\varphi_{\a_i}\|_{L^1(\Omega)}
\\
&=&\displaystyle  \max_{s\in [0,\|v_{0h}\|_{L^\infty(\Omega)}]} \Phi(s)  (\psi^{n+1}_h, u^{n+1}_h)_h
\\
&=&\displaystyle
\max_{s\in [0,\|v_{0h}\|_{L^\infty(\Omega)}]} \Phi(s) \|\psi^{n+1}_h\|^2_{H^1_h(\Omega)},
\end{array}
$$
where the last equality is obtained by taking $\bar\psi_h=\psi^{n+1}_h$ in \eqref{eq:psi_h}. Therefore,
\begin{equation}\label{lm6:lab2}
\begin{array}{rcl}
\|\psi^{n+1}_h\|^2_{H^1_h(\Omega)}-\|\psi^n_h\|^2_{H^1_h(\Omega)}+\|\psi^{n+1}_h-\psi^n_h\|^2_{H^1_h(\Omega)}
\\
+ 2 k(\Phi(v^n_h)u^{n+1}_h, u^{n+1}_h)_h\le 2 k \max_{x\in [0,\|v_{0h}\|_{L^\infty(\Omega)}]} \Phi(x) \|\psi^{n+1}_h\|^2_{H^1_h(\Omega)}.
\end{array}
\end{equation}
From \eqref{smallness:k}, Grönwall's inequality leads to 
$$
\begin{array}{rcl}
\displaystyle
\max_{n\in\{0, \cdots, N-1\}}\|\Psi^{n+1}_h\|^2_{H^1_h(\Omega)}&+&\displaystyle
k\sum_{n=0}^{N-1} (k^2 \|\delta_t\Psi^{n+1}_h\|_{H^1_h(\Omega)}^2+(\Phi(v^n_h)u^{n+1}_h, u^{n+1}_h)_h)
\\
&&\le e^{ \displaystyle2 T \max_{s\in [0,\|v_{0h}\|_{L^\infty(\Omega)}]} \Phi(s)} \|\psi^0_h\|^2_{H^1_h(\Omega)} 
\\
&\le& e^{\displaystyle 2 T \max_{s\in [0,\|v_{0h}\|_{L^\infty(\Omega)}]} \Phi(s)} \Big(C_{\rm com} C_{\rm inv}\|u^0_h\|_{L^1(\Omega)}+\|u^0_h\|_{(H^1(\Omega)'}\Big).
\end{array}
$$
This last inequality makes use of \eqref{eq:psi_h} and \eqref{com_err_L1->L2-H1:I_h} together with \eqref{inv_ineq_Wnp->Wmq}
and \eqref{equiv:L2_and_L2h} to find 
$$
\|\psi^0_h\|^2_{H^1_h(\Omega)}=(u^0_h, \psi^0_h)_h\le C_{\rm com} C_{\rm inv} \|u^0_h\|_{L^1(\Omega)} \, \|\nabla\psi^0_h\|_{L^2(\Omega)}+\|u^0_h\|_{(H^1(\Omega))'} \|\psi^0_h\|_{H^1_h(\Omega)},
$$
from which 
$$
\|\psi^0_h\|_{H^1_h(\Omega)}\le C_{\rm com}C_{\rm inv}\|u^0_h\|_{L^1(\Omega)}+\|u^0_h\|_{(H^1(\Omega)'}. 
$$
\end{proof}
\begin{lemma} There is a constant $C_2=C_2(v_0)>0$, independent of $(h,k)$, such that 
\begin{equation}\label{bnd:v^n_h} 
\max_{n\in\{0,\cdots, N\}} \|v^n_h\|^2_h +k\sum_{n=0}^{N-1} (k^2 \|\delta_t  v^{n+1}_h\|_h^2+ \|\nabla v^{n+1}_h\|^2_{L^2(\Omega)})\le C_2.
\end{equation}
\end{lemma}
\begin{proof}
For $\bar v_h=v^{n+1}_h$ in $\eqref{Alg}_2$, we find
$$
\|v^{n+1}_h\|^2_h-\|v^n_h\|^2_h+\|v^{n+1}_h-v^n_h\|^2_h+\|\nabla v^{n+1}_h\|^2_{L^2(\Omega)}+\|\sqrt{u^{n+1}_h}v^{n+1}_h\|^2_h=0.
$$
Summing over $n$ completes the proof. 
\end{proof}
\begin{lemma} There are two constants $C_3=C_3(u_0, v_0, T)>0$ and $C_4=C_4(u_0, v_0, T)>0$, independent of $(h,k)$, such that 
\begin{equation}\label{bnd: log v^n_h in Linf-L^1}
\max_{n\in\{0,\cdots, N\} }- \int_\Omega \mathcal{I}_h[\log v^n_h]\,{\rm d}\x \le C_3
\end{equation}
and
\begin{equation}\label{bnd: log v^n_n in L2H1}
k\sum_{n=0}^{N-1} \|\nabla\mathcal{I}_h\log v^{n+1}_h\|_{L^2(\Omega)}\le C_4.
\end{equation}

\end{lemma}
\begin{proof}
Let us now test  $\eqref{Alg}_2$ against $\bar v_h=\mathcal{I}_h[\frac{1}{v^{n+1}_h}]$, which is well-posed from \eqref{Positivity_and_DMP:v^n_h}, to obtain
\begin{equation}\label{lm8:lab1}
(\delta_t v^{n+1}_h, \frac{1}{v^{n+1}_h})_h+(\nabla v^{n+1}_h, \nabla\mathcal{I}_h[\frac{1}{v^{n+1}_h}])+\|u^{n+1}_h\|_{L^1(\Omega)}=0.
\end{equation} 
We first note using Taylor's formula that 
$$
\log v^n_h-\log v^{n+1}_h=\frac{1}{v^{n+1}_h} (v^n_h-v^{n+1}_h) - \frac{1}{(v^{n+\theta}_h)^2} (v^n_h-v^{n+1}_h)^2,
$$  
which implies 
\begin{equation}\label{lm8:lab2}
\frac{1}{v^{n+1}_h} (v^{n+1}_h-v^n_h)= \log v^{n+1}_h-\log v^n_h - \frac{1}{(v^{n+\theta}_h)^2} (v^n_h-v^{n+1}_h)^2,
\end{equation}
where $v^{n+\theta}_h=\theta v^{n+1}_h+(1-\theta) v^{n+1}_h$ with $\theta\in (0,1)$. Secondly, we write and note, from the acuteness of $\mathcal{T}_h$, that  
$$
(\nabla v^{n+1}_h, \nabla\mathcal{I}_h[\frac{1}{v^{n+1}_h}])=\sum_{i<j\in I}\frac{(v^{n+1}_h(\a_j)- v^{n+1}_h(\a_j))^2 }{v^{n+1}_h(\a_i) v^{n+1}_h(\a_j)} (\nabla\varphi_{\a_j}, \nabla\varphi_{\a_j})\le0.
$$
Further, the inequality $(\log x -\log y)^2\le \frac{(x-y)^2}{ x y} $ for $x>0$ and $y>0$ in \cite[Prop 3.7]{GS_2025} gives  
\begin{equation}\label{lm8:lab3}
-(\nabla v^{n+1}_h, \nabla\mathcal{I}_h[\frac{1}{v^{n+1}_h}])\ge \|\nabla\mathcal{I}_h[\log v^{n+1}_h]\|^2_{L^2(\Omega)}.
\end{equation}
Thus, combining \eqref{lm8:lab2} and \eqref{lm8:lab3} into \eqref{lm8:lab1} leads to  
$$
\begin{array}{rcl}
\displaystyle
-\int_\Omega\mathcal{I}_h[\log v^{n+1}_h]\, {\rm d}\x&+&\displaystyle k \|\frac{1}{v^{n+\theta}_h}\delta_t v^{n+1}_h\|_h^2
\\
&+&\displaystyle
k\|\nabla\mathcal{I}_h \log v^{n+1}_h\|^2_{L^2(\Omega)}\le -\int_\Omega \mathcal{I}_h[\log v^n_h]\,{\rm d}\x
+k \|u_{0h}\|_{L^1(\Omega)},
\end{array}
$$
where we used \eqref{mass_conv:u^n_h}. Hence, from the fact that
$$
\begin{array}{rcl}
\displaystyle
\int_\Omega \mathcal{I}_h[\log v^{n+1}_h]\,{\rm d}\x&=&\displaystyle \sum_{i\in I} \log v^{n+1}_h(\a_i)\|\varphi_{\a_i}\|_{L^1(\Omega)}
\\
&\le&\displaystyle \sum_{i\in I} \log \|v_{0h}\|_{L^\infty(\Omega)}\|\varphi_{\a_i}\|_{L^1(\Omega)}
\\
&=&
|\Omega| \log\|v_{0h}\|_{L^\infty(\Omega)},
\end{array}
$$  
due to \eqref{Positivity_and_DMP:v^n_h}, we infer, after summing over $n=0, \cdots, N-1$, that
$$
\begin{array}{rcl}
\displaystyle
k\sum_{n=0}^{N-1} \|\nabla\mathcal{I}_h\log v^{n+1}_h\|_{L^2(\Omega)}&\le&\displaystyle
T \|u_{0h}\|_{L^1(\Omega)}-\int_\Omega \mathcal{I}_h[\log v_{0h}]\,{\rm d}\x+ |\Omega| \log \|v_{0h}\|_{L^\infty(\Omega)}
\\
&:=&C_3
\end{array}
$$ 
and, furthermore, that 
$$
\displaystyle
\max_{n\in\{0,\cdots, N\} } - \int_\Omega \mathcal{I}_h[\log v^n_h]\,{\rm d}\x \le T \|u_{0h}\|_{L^1(\Omega)}-\int_\Omega\mathcal{I}_h[\log v_{0h}]\,{\rm d}\x:=C_4. 
$$
This completes the proof of the lemma. 
\end{proof}

\begin{lemma} Let $p>0$. Then there exists a constant $C_5=C_5(u_0, v_0, p, T)$, independent of $(h,k)$, such that
\begin{equation}\label{bnd:(inv_v^n+1_h)^p in L1logL1}
k\sum_{n=0}^{N-1}\left[\log \|\mathcal{I}_h[(v^{n+1}_h)^{-p}]\|_{L^1(\Omega)}\right]_+\le C_5.
\end{equation}
\end{lemma}
\begin{proof}
We recall from \cite{Bonilla_GS_2024} that, for each $\varphi\in C(\bar\Omega)$ with $\varphi\ge 0$, there holds the following Moser-Trudinger discrete functional inequality for polygonal domains: 
\begin{equation}\label{ineq:M-T}
\int_\Omega \mathcal{I}_h e^{\mathcal{I}_h\varphi(\x)}\, {\rm d}\x\le c_1 (1+\|\nabla \mathcal{I}_h\varphi\|^2_{L^2(\Omega)}) e^{c_2 \|\nabla \mathcal{I}_h\varphi\|^2_{L^2(\Omega)}+ c_2\|\mathcal{I}_h\varphi\|_{L^1(\Omega)}},
\end{equation}
with $c_1, c_2$ being two positive constants independent of $h$.

Let 
$$
\mathscr{N}_{h,k}=\{n\in\{0,\cdots, N-1\}\,:\, \int_\Omega \mathcal{I}_h[(v^{n+1}_h)^{-p}]\, {\rm d}\x>1\}.
$$
Then, write 
\begin{equation}\label{lm9:lab1}
k\sum_{n=0}^{N-1}\left[\log \int_\Omega \mathcal{I}_h[(v^{n+1}_h)^{-p}]\, {\rm d}\x \right]_+= k\sum_{n\in\mathscr{N}_{h,k}}\left(\log \int_\Omega \mathcal{I}_h[(v^{n+1}_h)^{-p}]\, {\rm d}\x \right).
\end{equation}
In order to obtain an estimate for \eqref{lm9:lab1}, we take $\varphi= - p \log\frac{v^{n+1}_h}{\|v_{0h}\|_{L^\infty(\Omega)}}$, being non-negative from \eqref{Positivity_and_DMP:v^n_h}, in \eqref{ineq:M-T} to get, by
$$
- p\mathcal{I}_h\log \frac{v^{n+1}_h}{\|v_{0h}\|_{L^\infty(\Omega)}}=-p \mathcal{I}_h\log v^{n+1}_h+p \log \|v_{0h}\|_{L^\infty(\Omega)},
$$
that, for all $n\in\mathscr{N}_{h,k}$,   
\begin{align*}
\log \int_\Omega \mathcal{I}_h e^{- p \mathcal{I}_h\log \frac{v^{n+1}_h}{\|v_{0h}\|_{L^\infty(\Omega)}}} {\rm d}\x = \log  \int_\Omega \mathcal{I}_h[\frac{1}{(v^{n+1}_h)^p}] {\rm d}\x+ p \log \|v_{0h}\|_{L^\infty(\Omega)}
\\
\le \log c_1 +\log (1+p^2 \|\nabla\mathcal{I}_h[\log\frac{v^{n+1}_h}{\|v_{0h}\|_{L^\infty(\Omega)}}]\|^2_{L^2(\Omega)}) 
\\
+c_2 p^2 \|\nabla \mathcal{I}_h[\log\frac{v^{n+1}_h}{\|v_{0h}\|_{L^\infty(\Omega)}}]\|^2_{L^2(\Omega)}+ c_2 p\|-\mathcal{I}_h[\log\frac{v^{n+1}_h}{\|v_{0h}\|_{L^\infty(\Omega)}}]\|_{L^1(\Omega)}.
\end{align*}
It is straightforward to see that 
$$
\|\nabla\mathcal{I}_h\log\frac{v^{n+1}_h}{\|v_{0h}\|_{L^\infty(\Omega)}}\|^2_{L^2(\Omega)}=\|\nabla\mathcal{I}_h\log v^{n+1}_h\|^2_{L^2(\Omega)}
$$
and
$$
\begin{array}{rcl}
\displaystyle
\|-\mathcal{I}_h[\log\frac{v^{n+1}_h}{\|v_{0h}\|_{L^\infty(\Omega)}}]\|_{L^1(\Omega)}&=&\displaystyle\sum_{i\in I} -\log\frac{v^{n+1}_h(\a_i)}{\|v_{0h}\|_{L^\infty(\Omega)}} \|\varphi_{\a_i}\|_{L^1(\Omega)}
\\
&=&\displaystyle\sum_{i\in I} \Big(-\log v^{n+1}_h(\a_i) + \log \|v_{0h}\|_{L^\infty(\Omega)} \Big) \|\varphi_{\a_i}\|_{L^1(\Omega)}
\\
&=&\displaystyle
\log \|v_{0h}\|_{L^\infty(\Omega)} |\Omega| - \int_\Omega\mathcal{I}_h[\log v^{n+1}_h]\,{\rm d}\x.
\end{array}
$$
Therefore,
$$
\begin{array}{rcl}
\log\|\mathcal{I}_h[\frac{1}{(v^{n+1}_h)^p}]\|_{L^1(\Omega)}&\le&\log c_1 +  (1+ c_2) p^2(1+\|\nabla\mathcal{I}_h[\log v^{n+1}_h]\|^2_{L^2(\Omega)} )
\\
&&\displaystyle
- c_2 p \int_\Omega\mathcal{I}_h[\log v^{n+1}_h]\,{\rm d}\x+ p(c_2|\Omega|-1)\log\|v_{0h}\|_{L^\infty(\Omega)}.
\end{array}
$$
Next, multiplying through by $k$ and summing for $n=0, \cdots, N-1$,  we have, from \eqref{bnd: log v^n_h in Linf-L^1} and \eqref{bnd: log v^n_n in L2H1}, that
$$
\begin{array}{rcl}
\displaystyle
k \sum_{n=0}^{N-1} \Big[\log\|\mathcal{I}_h[\frac{1}{(v^{n+1}_h)^p}]\|_{L^1(\Omega)}\Big]_+&\le&\log c_1+ (1+c_2)p^2 (T+ C_4)+c_2 p C_3
\\
&&+ p(c_2|\Omega|-1)\log\|v_{0h}\|_{L^\infty(\Omega)}:=C_5.
\end{array}
$$
\end{proof}
\begin{corollary} It follows that 
\begin{equation}\label{bnd: (v_h^n)^alpha/2 u_n^n+1 in L2L2}
k\sum_{n=0}^{N-1} \|(v_h^n)^\frac{\alpha}{2} u_h^{n+1}\|^2_h\le \frac{C_1}{\tilde C}. 
\end{equation}

\end{corollary}
\begin{proof}
We have, by \eqref{cond:Phi} and \eqref{Positivity_and_DMP:v^n_h}, that there exist $\tilde C>0$, independent of $(h,k)$, such that, for all  $ n\in\{0,\cdots, N\}$,
$$
\Phi(v^n_h)\ge \tilde C (v^n_h)^\alpha\quad\mbox{ in } \Omega . 
$$
Consequently, in view of \eqref{bnd:Phi(v^n_h)u^(n+1)_h in L2L2h}, we easily find that \eqref{bnd: (v_h^n)^alpha/2 u_n^n+1 in L2L2} holds.

\end{proof}

\section{Weak and strong precompactness properties}\label{sec:compactness}
Recalling definitions \eqref{def:w_hk_tilde}, \eqref{def:w_hk_hat}, and \eqref{def:w_hk}, we state the following weak convergences.

\begin{lemma}
The sequence
\begin{equation}\label{bnd_Linf:v_hk_tilde}
\{\widetilde v_{h,k}\}_{h,k > 0} \quad \text{is bounded in} \quad L^\infty(\Omega_T),
\end{equation}
and 
\begin{equation}\label{bnd:v_hk_hat in L2H1}
\{\widehat v_{h,k}\}_{h,k > 0} \quad \text{is bounded in} \quad L^\infty(\Omega_T) \cap L^2(0,T; H^1(\Omega)),
\end{equation}
and therefore admits a subsequence (still denoted by $\{v_{h,k}\}_{h,k > 0}$) that converges to a limit function 
$v \in L^\infty(\Omega_T) \cap L^2(0,T; H^1(\Omega))$. Specifically, as $h, k \to 0$, we have
\begin{equation}\label{w_conv_L2H1:vh->v}
\widehat v_{h,k} \rightharpoonup v \quad \text{in} \quad L^2(0,T; H^1(\Omega)).
\end{equation}
\end{lemma}

\begin{proof} It is easy to see from bounds \eqref{Positivity_and_DMP:v^n_h} and \eqref{bnd:v^n_h} that \eqref{bnd_Linf:v_hk_tilde} and \eqref{bnd:v_hk_hat in L2H1} hold. Consequently, this implies \eqref{w_conv_L2H1:vh->v} from the weak topology on $L^2(0,T; H^1(\Omega))$. 
\end{proof}

The weak convergence of the sequence $\{\widehat u_{h,k}\}_{h,k>0}$ requires more than mere boundedness in $L^1(\Omega_T)$ from \eqref{mass_conv:u^n_h}; it involves a more refined analysis that makes use of the Dunford–Pettis theorem on bounded domains.
\begin{lemma}\label{lm:equi-integrability_u_hk}
The sequence
\begin{equation}\label{bnd:u_hk L1L1}
\{\widehat u_{h,k}\}_{h,k > 0} \quad \text{is bounded in} \quad L^\infty(0,T; L^1(\Omega)).
\end{equation}
Furthermore, there exists a subsequence of $\{\widehat u_{h,k}\}_{h,k > 0}$ (not relabeled) and a function $u \in L^1(\Omega_T)$ such that
\begin{equation}\label{w_conv_L1L1:u_hk->u}
\widehat u_{h,k} \rightharpoonup u \quad \text{in} \quad L^1(\Omega_T)\quad\mbox{ as }\quad h, k \to 0.
\end{equation}

\end{lemma}\begin{proof}  It follows directly from bound \eqref{mass_conv:u^n_h} that \eqref{bnd:u_hk L1L1} holds. Moreover, from \eqref{bnd:(inv_v^n+1_h)^p in L1logL1}, we obtain
\begin{equation}\label{bnd:(inv_tilde_v_hk)^p in L1logL1}
\int_k^{T}[\log \|\widetilde v_{h,k}^{-\frac{p}{2}}\|_h]_+\,{\rm d}t\le C_5.
\end{equation}

We now aim to show that the sequence $\{1_{(0,T)\setminus(0,k]} \widehat u_{h,k}\}_{h,k>0}$ is equi-integrable over $\Omega_T$, where $1_{(0,T)\setminus(0,k]}$ is the indicator function of $(0,T)\setminus(0,k] $; that is, for any $\varepsilon > 0$, there exists $\delta > 0$ such that, for any measurable set $\mathscr{S} \subset \Omega_T$, if $|\mathscr{S}| < \delta$, then
$$
\|1_{(0,T)\setminus(0,k]}\widehat u_{h,k}\|_{L^1(\mathscr{S})} < \varepsilon \quad \text{for all } (h,k).
$$
Given a measurable set $\mathscr{S} \subset \Omega_T$, we define, for each $t \in (0,T)$, the spatial section
$$
\mathscr{S}(t) := \{ \x \in \Omega \,:\, (\x,t) \in \mathscr{S} \}.
$$
We then write
$$
\|\widehat u_{h,k}\|_{L^1(\mathscr{S})} = \int_0^{T} \int_{\mathscr{S}(t)} \widehat u_{h,k}\, {\rm d}\x\, {\rm d}t.
$$
Now define the time set
$$
\mathscr{T}_{h,k} := \left\{ t \in (k,T) : \|\widetilde v_{h,k}^{-\alpha}(t)\|_h^2 \le K \right\},
$$
where $K > 1$ is a constant to be chosen later, and split the integral as
\begin{equation}\label{lm12:lab1}
\int_k^{T} \widehat u_{h,k} \, {\rm d}\x \, {\rm d}t =
\int_{\mathscr{T}_{h,k}} \int_{\mathscr{S}(t)} \widehat u_{h,k} \,{\rm d}\x \, {\rm d}t
+\int_{\mathscr{T}_{h,k}^c} \int_{\mathscr{S}(t)} \widehat u_{h,k} \, {\rm d}\x \, {\rm d}t.
\end{equation}
From the bound 
$$
\int_{\mathscr{T}_{h,k}^c} \log \left\| \mathcal{I}_h [\widetilde v_{h,k}^{-2\alpha}\right] \|_h {\rm d}t \ge |\mathscr{T}_{h,k}^c| \log K,
$$
we deduce, on using the upper bound $C_5$ in \eqref{bnd:(inv_tilde_v_hk)^p in L1logL1}, that
$$
|\mathscr{T}_{h,k}^c| \le \frac{C_5}{\log K}.
$$
Thus, the contribution to the $L^1$-norm from $\mathscr{T}_{h,k}^c$  is controlled by
\begin{equation}\label{lm12:lab2}
\int_{\mathscr{T}_{h,k}^c} \int_{\mathscr{S}(t)} \widehat u_{h,k} {\rm d}\x\, {\rm d}t \le \|u_{0h}\|_{L^1(\Omega)} |\mathscr{T}_{h,k}^c|< \frac{\varepsilon}{2},
\end{equation}
by choice of $K> 1$ large enough such that 
$$
\frac{\|u_{0h}\|_{L^1(\Omega)}C_5}{\log K}<\frac{\varepsilon}{2}.
$$

Using the finite element structure, we have, for each $t \in \mathscr{T}_{h,k}$, that
\begin{align*}
\int_{\mathscr{S}(t)} \widehat u_{h,k}\, {\rm d}\x
= &\sum_{i \in I}  u_{h,k}(\a_i) \int_{\mathscr{S}(t)} \varphi_{\a_i} {\rm d}\x
= \sum_{i \in I}  \frac{ \widetilde v_{h,k}^\frac{\alpha}{2}(\a_i)}{\widetilde v_{h,k}^{\frac{\alpha}{2}}(\a_i)}  u_{h,k}(\a_i)\int_{\mathscr{S}(t)} \varphi_{\a_i}{\rm d}\x \\
\le& \left( \sum_{i \in I} \widetilde v_{h,k}^\alpha (\a_i)\widehat u_{h,k}^2(\a_i) \int_{\mathscr{S}(t)} \varphi_{\a_i}\,{\rm d}\x \right)^{\frac{1}{2}}
\\
&\times\left(\sum_{i \in I} \frac{1}{\widetilde v^{2\alpha}_{h,k}(\a_i)} \int_{\mathscr{S}(t)} \varphi_{\a_i}\,{\rm d}\x \right)^{\frac{1}{4}} \left(\sum_{i\in I}\int_{\mathscr{S}(t)}\varphi_{\a_i}\,{\rm d}\x\right)^\frac{1}{4}.
\end{align*}
This leads to:
$$
\int_{\mathscr{T}_{h,k}}\int_{\mathscr{S}(t)} \widehat u_{h,k} \, {\rm d}\x \, {\rm d}t
\le \int_{\mathscr{T}_{h,k}}  \| \widetilde v_{h,k}^\frac{\alpha}{2}  \widehat u_{h,k}\|_h
\|\widetilde v^{-\alpha}_{h,k}\|_h^{\frac{1}{2}} |\mathscr{S}(t)|^{\frac{1}{4}} {\rm d}t.
$$
In view of the definition of $\widehat\mathscr{T}_{h,k}$ and $\widetilde\mathscr{T}_{h,k}$, we have
$$
\int_{\mathscr{T}_{h,k}} \int_{\mathscr{S}(t)} \widehat  u_{h,k}\, {\rm d}\x\, {\rm d}t
\le K^{\frac{1}{4}} \int_0^{T}\| \widetilde v_{h,k}^\frac{\alpha}{2} \widehat  u_{h,k}\|_h  |\mathscr{S}(t)|^{\frac{1}{2}} {\rm d}t,
$$
which implies that 
$$
\int_{\mathscr{T}_{h,k}} \int_{\mathscr{S}(t)}  u_{h,k}\, {\rm d}\x\, {\rm d}t
\le K^{1/4} 
\left( \int_0^{T} \|\widetilde v_{h,k}^\frac{\alpha}{2} \widehat u_{h,k}^2\|^2_h\,{\rm d}t \right)^{\frac{1}{2}} \left(\int_0^{T} |\mathscr{S}(t)|  {\rm d}t \right)^{\frac{1}{4}} T^\frac{1}{4}.
$$
Thus, we conclude, from \eqref{bnd: (v_h^n)^alpha/2 u_n^n+1 in L2L2}, that
\begin{equation}\label{lm12:lab3}
\int_{\mathscr{T}_{h,k}} \int_{\mathscr{S}(t)}  u_{h,k}\, {\rm d}\x\, {\rm d}t
\le K^{\frac{1}{4}} T^{\frac{1}{4}} \left( \frac{C_2}{\tilde{C}} \right)^{\frac{1}{2}} |\mathscr{S}|^{\frac{1}{4}}<\frac{\varepsilon}{2}
\end{equation}
holds whenever $|\mathscr{S}| < \delta$  and $\delta > 0$ is chosen such that
$$
K^{\frac{1}{4}} T^{\frac{1}{4}} \left( \frac{C_2}{\tilde{C}} \right)^{\frac{1}{4}} \delta^{\frac{1}{4}} < \frac{\varepsilon}{2}.
$$
Compiling \eqref{lm12:lab2} and \eqref{lm12:lab3} into \eqref{lm12:lab1} yields 
$$
\| 1_{(0,T)\setminus(0,k]}  \widehat u_{h,k}\|_{L^1(\mathscr{S})}<\varepsilon.
$$

Finally, we verify \eqref{w_conv_L1L1:u_hk->u}. To carry this out, let $\varphi\in L^\infty(\Omega_T)$ and write
$$
\int_{\Omega_T} \widehat u_{h,k}\varphi\,{\rm d}\x\,{\rm d}t= \int_0^k \int_\Omega \widehat u_{h,k}\varphi\,{\rm d}\x\,{\rm d}t+\int_{\Omega_T}1_{(0,T)\setminus(0,k]} \widehat u_{h,k}\varphi\,{\rm d}\x\,{\rm d}t.
$$
We now know that there is $u\in L^1(\Omega_T)$ such that 
$$
\int_{\Omega_T}1_{(0,T)\setminus(0,k]} \widehat u_{h,k}\varphi\,{\rm d}\x\,{\rm d}t  \to \int_{\Omega_T} u\varphi\,{\rm d}\x\,{\rm d}t. 
$$
Furthermore, from \eqref{mass_conv:u^n_h}, we see that
$$
 \int_0^k \int_\Omega \widehat u_{h,k}\varphi\,{\rm d}\x\,{\rm d}t\le k \|\varphi\|_{L^\infty(\Omega_T)} \|u_{0h}\|_{L^1(\Omega)}\quad\mbox{ as }\quad h,k\to 0.  
$$ 
\end{proof}

The proof of the following corollary follows from the fact that $\|\Phi(\widetilde{v}_{h,k})\|_{L^\infty(\Omega)}\le\break \max_{x\in[0,\|v_{0h}\|_{L^\infty(\Omega)}\|]}\Phi(x)$ for all $(h,k)$. 
\begin{corollary}
The sequence $\{\Phi(\widetilde{v}_{h,k})\, \widehat u_{h,k}\}_{h,k>0}$ is equi-integrable on $\Omega_T$ and, as a result, satisfies
\begin{equation}\label{w_conv_L1L1:Ph(v_hk_hat)u_hk->z}
\Phi(\widetilde{v}_{h,k})\, \widehat u_{h,k} \rightharpoonup z \quad \text{in } L^1(\Omega_T)\quad\mbox{ as }\quad h,k\to 0,
\end{equation}
for some $z \in L^1(\Omega_T)$.
\end{corollary}

Next we establish the strong convergence of the sequence $\{\widehat v_{h,k}\}_{h,k}$ needed to pass to the limit in the term $(\nabla\mathcal{I}_h[\Phi(\widehat v_{h,k}) u_{h,k}], \nabla \bar u_h)$ as $h,k\to 0$.
\begin{lemma} We have 
\begin{equation}\label{s_conv: v_hk_hat in L2L2}
\widehat v_{h,k}\to v\quad\mbox{ in } L^2(0,T; L^2(\Omega))\quad\mbox{ as }\quad h,k\to 0.
\end{equation}
\end{lemma}

\begin{proof} We aim to prove that for each $0<\delta < T$, the following inequality holds:
$$
\int_0^{T-\delta} \|\widehat v_{h,k}(t+\delta)- \widehat v_{h,k}(t)\|_{L^2(\Omega)}\,{\rm d}t \le C\, \delta^{\frac{1}{2}},
$$
for $C=C(u_0,v_0, T)$ being independent of $(h,k)$.  Since $\widetilde v_{h,k}$ is piecewise constant in time, it suffices to consider values of $\delta$ that are integer multiples of the time step $k$, that is, $\delta = r\,k$ with $r = 1, \dots, N-1$. Thus, it is enough to prove \cite{GG_GS_2008} that
$$
k \sum_{m=0}^{N-r-1} \|v^{m+r+1}_h - v^{m+1}_h\|^2_{L^2(\Omega)} \le C\,(r\,k)^{1/2},\quad \forall\, r: 0 \leq r \leq N-1.
$$

To derive this estimate, we begin by multiplying  equation $\eqref{Alg}_2$ by $k$ and summing over $n = m+1, \dots, m+r$. This yields
$$
( v_h^{m+r+1} - v_h^{m+1}, \bar v_h )_h + k \sum_{n=m+1}^{m+r} (\nabla v^{n+1}_h, \nabla \bar v_h) + k \sum_{n=m+1}^{m+r} (u^{n+1}_h v^{n+1}_h, \bar v_h)_h = 0.
$$
Choosing the test function $\bar v_h = v^{m+r+1}_h - v^{m+1}_h$, we obtain
$$
\|v^{m+r+1}_h - v^{m+1}_h\|^2_h = -k \sum_{n=m+1}^{m+r} (\nabla v^{n+1}_h, \nabla(v^{m+r}_h - v^m_h)) - k \sum_{n=m+1}^{m+r} (u^{n+1}_h v^{n+1}_h, v^{m+r}_h - v^m_h)_h.
$$

Multiplying both sides by $k$ and summing over $m = 0, \dots, N - r-1$ gives
$$
\begin{aligned}
k\sum_{m=0}^{N-r-1} \|v^{m+r+1}_h - v^{m+1}_h\|^2_h &= -k^2 \sum_{m=0}^{N-r-1} \sum_{n=m+1}^{m+r} (\nabla v^{n+1}_h, \nabla(v^{m+r+1}_h - v^{m+1}_h)) \\
&\quad - k^2 \sum_{m=0}^{N-r-1} \sum_{n=m}^{m+r-1} (u^{n+1}_h v^{n+1}_h, v^{m+r+1}_h - v^{m+1}_h)_h.
\end{aligned}
$$

We now estimate the two terms on the right-hand side. A discrete version of Fubini's theorem allows us to switch the order of summation:
$$
\begin{aligned}
-k^2 \sum_{m=0}^{N-r-1} \sum_{n=m+1}^{m+r} (\nabla v^{n+1}_h, \nabla(v^{m+r+1}_h - v^{m+1}_h))
&= -k^2 \sum_{n=0}^{N-2} \sum_{m=\overline{n-r+1}}^{\bar n} (\nabla v^{n+1}_h, \nabla(v^{m+r+1}_h - v^{m+1}_h)) \\
&\le k \sum_{n=0}^{N-2} \|\nabla v^{n+1}_h\|_{L^2(\Omega)} \\
&\quad\times\left(k \sum_{m=\overline{n-r+1}}^{\bar n} \|\nabla(v^{m+r+1}_h - v^{m+1}_h)\|^2_{L^2(\Omega)}\right)^{\frac{1}{2}} \\
&\quad \times \left( k \sum_{m=\overline{n-r+1}}^{\bar n} 1 \right)^{\frac{1}{2}} \\
&\le 2 T^{1/2} (r k)^{\frac{1}{2}} \sum_{n=0}^{N-1} k \|\nabla v^{n+1}_h\|^2_{L^2(\Omega)} \\
&\le 2 T^{1/2} C_2 (r k)^{\frac{1}{2}},
\end{aligned}
$$
where the index cutoff is defined by
$$
\bar n = \begin{cases}
0 & \text{if } n < 0, \\
n & \text{if } 0 \le n \le N - r-1, \\
N - r-1 & \text{if } n > N - r-1,
\end{cases}
$$
ensuring $|\bar n - \overline{n - r + 1}| \le r$. In this last inequality we used \eqref{bnd:v^n_h}.  

Similarly, using estimates such as \eqref{Positivity_and_DMP:v^n_h} and \eqref{mass_conv:u^n_h}, the second term is bounded by
$$
\begin{aligned}
-k^2 \sum_{m=0}^{N-r-1} \sum_{n=m+1}^{m+r+1} (u^{n+1}_h v^{n+1}_h, v^{m+r+1}_h - v^{m+1}_h)_h &= -k^2 \sum_{n=0}^{N-2} \sum_{m=\overline{n-r+1}}^{\bar n} (u^{n+1}_h v^{n+1}_h, v^{m+r}_h - v^m_h)_h \\
&\le 2 \|v_{0h}\|^2_{L^\infty(\Omega)}\, k \sum_{n=0}^{N-1} \|u^{n+1}_h\|_{L^1(\Omega)}\, r k \\
&\le 2 T^\frac{3}{2} \|v_{0h}\|^2_{L^\infty(\Omega)} \|u_{0h}\|_{L^1(\Omega)} (r k)^{\frac{1}{2}}.
\end{aligned}
$$

Combining these estimates yield, on noting \eqref{equiv:L2_and_L2h}, that
$$
k \sum_{m=0}^{N-r-1} \|v^{m+r+1}_h - v^{m+1}_h\|^2_{L^2(\Omega)} \le C (r k)^{\frac{1}{2}}.
$$

The strong convergence \eqref{s_conv: v_hk_hat in L2L2} follows from Simon’s compactness theorem \cite{Simon_1987} and boundedness \eqref{bnd:v_hk_hat in L2H1} entails \eqref{s_conv: v_hk_hat in L2L2}.

\end{proof}
\begin{corollary} There holds
\begin{equation}\label{s_conv: v_hk_tilde in L2L2}
\widetilde v_{h,k}\to v\quad\mbox{ in } L^2(\Omega_T)\quad\mbox{ as }\quad h,k\to  0.
\end{equation}
\end{corollary}
\begin{proof}
On using \eqref{bnd:v^n_h}, there results 
$$
\begin{array}{rcl}
\displaystyle
\int_0^T \|\widehat v_{h,k}-\widetilde v_{h,k}\|^2_{L^2(\Omega)}\, {\rm d}t&\le&\displaystyle
C_{\rm sta}\int_0^T \|\widetilde v_{h,k}-\widehat v_{h,k}\|^2_h\, {\rm d}t
\\	
&=&\displaystyle
C_{\rm sta}  k \sum_{n=0}^{N-1} k^2 \|\delta_t v^{n+1}_h\|^2_h
\\
&\le& k C_{\rm sta} C_2\to 0\quad\mbox{ as } h,k\to 0.
\end{array}
$$
Thus, \eqref{s_conv: v_hk_hat in L2L2} is clearly satisfied.     
\end{proof}

In the following lemma we identify the limit function $z=\Phi(v) u$ in \eqref{w_conv_L1L1:Ph(v_hk_hat)u_hk->z}.

\begin{lemma}
It follows that  
\begin{equation}\label{w_conv_L1L1:Ph(v_hk_hat)u_hk->Phi(v)u}
\mathcal{I}_h[\Phi(\widetilde v_{h,k}) \widehat u_{h,k}] \rightharpoonup \Phi(v) u \quad\mbox{ in } L^1(\Omega_T)\quad\mbox{ as }\quad h,k\to0.
\end{equation}
\end{lemma}
\begin{proof}
We first infer from \eqref{s_conv: v_hk_tilde in L2L2} that 
$$
\widetilde v_{h,k}(\x,t)\to v(\x, t)\quad\mbox{a. e. on }\Omega_T. 
$$
As a result, by continuity of $\Phi$, we also have
$$
\Phi(\widetilde v_{h,k}(\x,t))\to \Phi(v(\x, t))\quad\mbox{a. e. on }\Omega_T. 
$$
Then, one can prove \cite[Lm A.1]{Zhigun_Surulescu_Uatay_2016}, on nothing \eqref{w_conv_L1L1:u_hk->u} and \eqref{w_conv_L1L1:Ph(v_hk_hat)u_hk->z}, that 
$$
\Phi(\widetilde v_{h,k}) \widetilde u_{h,k} \rightharpoonup \Phi(v) u \quad\mbox{ in } L^1(\Omega_T)\quad\mbox{ as }\quad h,k\to 0.
$$ 
Our next step to completing the proof of \eqref{w_conv_L1L1:Ph(v_hk_hat)u_hk->Phi(v)u} is to obtain  
\begin{equation}\label{lm16:lab1}
\mathcal{I}_h[\Phi(\widetilde v_{h,k})\widehat u_{h,k}]- \Phi(\widetilde v_{h,k})\widehat u_{h,k}  \to 0 \quad\mbox{ in } L^1(\Omega_{T})\quad\mbox{ as }\quad h,k\to0.
\end{equation}
Let $\Omega_{T^*} = \Omega \times (k, T)$ and consider a family of mollifiers $\{\rho_\lambda\}_{\lambda > 0}$. Define the mollified function
$$
\Phi_\lambda = \rho_\lambda \star \Phi,
$$
that is, the convolution of $\Phi$ with the mollifier $\rho_\lambda$. For each $n \in \mathds {N}$, denote by $\Phi^{(n)}_\lambda$ the $n$-th derivative of $\Phi_\lambda$, and define the constant
$$
c_{\Phi^{(n)}_\lambda} = \max_{x \in [0, \|v_{0h}\|_{L^\infty(\Omega)}]} \Phi^{(n)}_\lambda(x).
$$
It is clear that $c_{\Phi^{(n)}_\lambda}$ depends badly on $\lambda$ for $n\ge1$. Now write:
\begin{equation}\label{lm16:decomposition}
\begin{array}{rcl}
\mathcal{I}_h[\Phi(\widetilde v_{h,k})\widehat u_{u,k}]- \Phi(\widetilde v_{h,k})\widehat u_{h,k} &=& \mathcal{I}_h[\Phi(\widetilde v_{h,k}) \widehat u_{h,k})] \pm\mathcal{I}_h[\Phi_\lambda(\widetilde v_{h,k}) \widehat u_{h,k}] 
\\
&& \pm \mathcal{I}_h[\Phi_\lambda(\widetilde v_{h,k})]  \widehat u_{h,k}
\\
&&\pm\Phi_\lambda(\widetilde v_{h,k}) \widehat u_{h,k} - \Phi(\widetilde v_{h,k}) \widehat u_{h,k}
\\
&:=&\displaystyle
\sum_{i=1}^4 R_i.
\end{array}
\end{equation}
The right-hand side is estimated as follows:

$\bullet$ For $R_1$, we know \cite[Pro. 4.21]{Brezis_2011} that, for each $\varepsilon_1 >0$, there exists $\lambda_0>0$, independent of $(h,k)$, such that, for all $\lambda<\lambda_0$, it follows from \eqref{bnd_Linf:v_hk_tilde}  that   
$$
\|\Phi_\lambda (\widetilde v_{h,k}) - \Phi(\widetilde  v_{h,k})\|_{L^\infty(\Omega_{T^*})}\le \varepsilon_1.  
$$
Thus, by \eqref{mass_conv:u^n_h},  
\begin{equation}\label{lm16:R1->0}
\begin{array}{rcl}
\displaystyle
\|R_1\|_{L^1(\Omega_{T^*})}&=&\displaystyle k\sum_{n=1}^N 
\sum_{i\in I} |[\Phi(\widetilde v_{h,k}) - \Phi_\lambda(\widetilde v_{h,k})]|(\a_i) \widehat u_{h,k}(\a_i) \|\varphi_{\a_i}\|_{L^1(\Omega)}
\\
&\le&\displaystyle
\varepsilon_1 k\sum_{n=1}^N  \|\widehat u_{h,k}\|_{L^1(\Omega)}\le T \varepsilon_1 \|u_{0h}\|_{L^1(\Omega)}.
\end{array}
\end{equation}

$\bullet$ For $R_2$, in the basis of Chebyshev's inequality, we define, as $\nabla \widehat v_{h,k}$ is space-time piecewise continuous, for $M_1>0$, 
$$
\mathscr{A}=\{(\sigma,\tau) \in \Sigma_h \times\mathcal{T}^*_k: |\nabla\widetilde v_{h,k}|_{\tau\times\sigma} |>M_1\}
$$
and $\mathscr{A}^c=\Omega_{T^*}\backslash  \mathscr{A}$, where $\mathcal{T}^*_h=\mathcal{T}_h\backslash[0,k]$. Then, we have, by \eqref{bnd:v^n_h}, that 
$$
|\mathscr{A}|\le \frac{\|\nabla\widetilde v_{h,k}\|_{L^2(\Omega_{T^*})}^2}{M^2_1}\le \frac{C_2}{M^2_1}  .
$$
We express
$$
\|R_2\|_{L^1(\Omega_{T^*})}=\sum_{(\sigma,\tau)\in \mathscr{A} } \|R_2\|_{L^1(\sigma\times\tau)}+\sum_{(\sigma,\tau)\in \mathscr{A}^c} \|R_2\|_{L^1(\sigma\times\tau)}.
$$
For each $\tau\in\mathcal{T}_h$, we define 
$$
\mathscr{A}(\tau) := \{ \sigma \in \Sigma_h \,:\, (\sigma,\tau) \in \mathscr{A} \},
$$
and, analogously, we define $\mathscr{A}^c(\tau)$.  Next,  we invoke  \eqref{err_Linf-W2inf:I_h} and \eqref{inv_ineq_Wnp->Wmq} to get  
$$
\begin{array}{rcl}
\displaystyle
\sum_{(\sigma,\tau)\in \mathscr{A}} \|R_2\|_{L^1(\sigma\times\tau)}&=&\displaystyle
\sum_{\tau\in\mathcal{T}_k^*}\sum_{\sigma\in \mathscr{A}(\tau) } \|R_2\|_{L^1(\sigma\times\tau)} 
\\
&=&\displaystyle
k \sum_{\tau\in\mathcal{T}_k^*}\sum_{\sigma\in \mathscr{A}(\tau)} \|\mathcal{I}_h[\Phi_\lambda(\widetilde v_{h,k}) \widehat u_{h,k}] - \mathcal{I}_h[\Phi_\lambda(\widetilde v_{h,k})] \widehat u_{h,k}\|_{L^1(\sigma)}
\\
&\le&\displaystyle
k \sum_{\tau\in\mathcal{T}_k^*}\sum_{\sigma\in \mathscr{A}(\tau)} |\sigma| \|\mathcal{I}_h[\mathcal{I}_h[\Phi_\lambda(\widetilde v_{h,k})]\widehat u_{h,k}] - \mathcal{I}_h[\Phi_\lambda(\widetilde v_{h,k})] \widehat u_{h,k}\|_{L^\infty(\sigma)}
\\
&\le& \displaystyle
C_{\rm err} k \sum_{\tau\in\mathcal{T}_k}\sum_{\sigma\in \mathscr{A}(\tau)} h^2 |\sigma| \|\nabla^2 (\mathcal{I}_h[\Phi_\lambda(\widetilde v_{h,k})] \widehat u_{h,k})\|_{L^\infty(\sigma)}
\\
&\le & \displaystyle
C_{\rm err} k \sum_{\tau\in\mathcal{T}_k^*}\sum_{\sigma\in \mathscr{A}(\tau)} h^2 |\sigma| \|\nabla\mathcal{I}_h[\Phi_\lambda(\widetilde v_{h,k})]\|_{L^\infty(\sigma)}  \|\nabla\widehat u_{h,k}\|_{L^\infty(\sigma)}
\\
&\le & \displaystyle
C_{\rm err} k \sum_{\tau\in\mathcal{T}_k^*}\sum_{\sigma\in \mathscr{A}(\tau)} h^2  \|\nabla\Phi_\lambda(\widetilde v_{h,k})\|_{L^\infty(\sigma)}  \|\nabla \widehat u_{h,k}\|_{L^1(\sigma)}
\\
&\le & \displaystyle
C_{\rm inv}C_{\rm err}  k \sum_{\tau\in\mathcal{T}_k^*}\sum_{\sigma\in \mathscr{A}(\tau)}  h c_{\Phi_\lambda'} \|\nabla \widetilde v_{h,k}\|_{L^\infty(\sigma)} \|\widehat u_{h,k}\|_{L^1(\sigma)} 
\\
&\le & C^2_{\rm inv} C_{\rm err} c_{\Phi_\lambda'} \|v_{0h}\|_{L^\infty(\Omega)}  \|\widehat u_{h,k}\|_{L^1(\mathscr{A})}.
\end{array}
$$
Similarly, we have
$$
\begin{array}{rcl}
\displaystyle
\sum_{\sigma\times\tau\in \mathscr{A}^c} \|R_2\|_{L^1(\sigma\times\tau)}&\le&\displaystyle 
C_{\rm inv}C_{\rm err}  k \sum_{\tau\in\mathcal{T}_k}\sum_{\sigma\in \mathscr{A}^c(\tau)}  h c_{\Phi_\lambda'} \|\nabla v_{h,k}\|_{L^\infty(\sigma)} \|\widehat u_{h,k}\|_{L^1(\sigma)}
\\
&\le&\displaystyle 
C_{\rm inv} C_{\rm err} h c_{\Phi_\lambda'} M_1   k \sum_{\tau\in\mathcal{T}_k}\sum_{\sigma\in A^c(\tau)}  \|\widehat u_{h,k}\|_{L^1(\sigma)}
\\
\mbox{by \eqref{mass_conv:u^n_h}}&\le&\displaystyle 
C_{\rm inv}C_{\rm err} T  c_{\Phi_\lambda'} M_1  \|u_{0h}\|_{L^1(\Omega)} h.
\end{array}
$$
Consequently, 
\begin{equation}\label{lm16:lab2}
\begin{array}{rcl}
\|R_2\|_{L^1(\Omega_{T^*})}&\le& C^2_{\rm inv} C_{\rm err} c_{\Phi_\lambda'}  \|v_{0h}\|_{L^\infty(\Omega)} \|\widehat u_{h,k}\|_{L^1(\mathscr{A})} 
\\
&&+C_{\rm inv}C_{\rm err} c_{\Phi'_\lambda} \|u_{0h}\|_{L^1(\Omega)} T   M_1  h.
\end{array}
\end{equation}
By Lemma \ref{lm:equi-integrability_u_hk}, we have that $\{\widehat u_{h,k}\}_{h,k>0}$ is equi-integrable. That is, for each $\varepsilon_2>0$, there exists $\delta_2>0$ such that for any measurable set  $\mathscr{S}\subset\Omega_T$ with $|\mathscr{S}|<\delta_2$, it follows that 
$$
\|u_{h,k}\|_{L^1(\mathscr{S})}<\varepsilon_2\quad\mbox{ for all }\quad h,k>0.
$$
This way, by choosing $M_1$ sufficiently large  so that  $|\mathscr{A}|<\delta_2$,  the first term on the right-hand side of \eqref{lm16:lab2} is bounded as follows:
 \begin{equation}\label{lm16:R2->0}
 \begin{array}{rcl}
\|R_2\|_{L^1(\Omega_{T^*})}&\le& C^2_{\rm inv} C_{\rm err} c_{\Phi_\lambda'}  \|v_{0h}\|_{L^\infty(\Omega)} \varepsilon_2
\\
&&+ C_{\rm inv}C_{\rm err} c_{\Phi'_\lambda} \|u_{0h}\|_{L^1(\Omega)} T   M_1  h.
\end{array}
 \end{equation}

$\bullet$ For $R_3$, again, by Chebyshev's inequality, we define, for $M_2>0$,
$$
\mathscr{B}=\{(\x,t)\in\Omega_{T^*} :\, |\widehat u_{h,k}|>M_2\}
$$
and hence
$$
|\mathscr{B}|\le\frac{\|u_{0h}\|_{L^1(\Omega)}}{M_2}.
$$
Then we decompose 
$$
\|R_3\|_{L^1(\Omega_{T^*})}=\|R_3\|_{L^1(\mathscr{B})} + \|R_3\|_{L^1(\mathscr{B}^c)}.
$$
Thus, we have that 
$$
\begin{array}{rcl}
\displaystyle
\|R_3\|_{L^1(\mathscr{B})} &=&\displaystyle
\int_{\mathscr{B}} |(\mathcal{I}_h[\Phi_\lambda(\widetilde v_{h,k})] -\Phi_\lambda(\widetilde v_{h,k} )) \widehat u_{h,k}| \,{\rm d}\x\,{\rm d}t 
\\
&\le&\displaystyle
2 c_{\Phi_\lambda} \int_{\mathscr{B}} \widehat u_{h,k} \,{\rm d}\x\,{\rm d}t 
\\
&\le&2 c_{\Phi_\lambda}  \|\widehat u_{h,k}\|_{L^1(\mathscr{B})}.
\end{array}
$$
and, further, for each $\sigma\in\Sigma_h$,  by \eqref{err_L1-W21:I_h}, that 
$$
\begin{array}{rcl}
\displaystyle
\|R_3\|_{L^1(\mathscr{B}^c)}&=&\displaystyle
\int_{\mathscr{B}^c } |(\mathcal{I}_h[\Phi_\lambda(\widehat v_{h,k})] -\Phi_\lambda(\widehat v_{h,k}) ) \widehat u_{h,k}| \,{\rm d}\x\,{\rm d}t 
\\
&\le&\displaystyle
 M_2 \int_{\mathscr{B}^c} |\mathcal{I}_h[\Phi_\lambda(\widetilde v_{h,k})] -\Phi_\lambda(\widetilde v_{h,k})| \,{\rm d}\x\,{\rm d}t 
\\
&\le&\displaystyle
 M_2 \|\mathcal{I}_h[\Phi_\lambda(\widetilde v_{h,k})] -\Phi_\lambda(\widetilde v_{h,k})\|_{L^1(\Omega_{T^*})}
\\
 &= &\displaystyle
 M_2 k \sum_{n=1}^N \sum_{\sigma\in\Sigma_h} \|\mathcal{I}_h[\Phi_\lambda(v^n_h)] -\Phi_\lambda(v^n_h)\|_{L^1(\sigma)}
\\
&\le&\displaystyle
C_{\rm err}M_2 k \sum_{n=1}^N  \sum_{\sigma\in \Sigma_h} h^2 \|\nabla^2\Phi_\lambda(v^n_h)\|_{L^1(\sigma)}
\\
&\le&\displaystyle
C_{\rm err}M_2 c_{\Phi_\lambda''} k \sum_{n=1}^N\sum_{\sigma\in\Sigma_h } h^2  \|\nabla v^n_h\|^2_{L^2(\sigma)}
\\
&\le&C_{\rm err}M_2 c_{\Phi_\lambda''} h^2 \|\nabla \widehat v_{h,k}\|^2_{L^2(\Omega_{T})}.
\end{array}
$$
Lemma \ref{lm:equi-integrability_u_hk} shows that, for each $\varepsilon_3>0$, there exists $\delta_2>0$ such that for any measurable set  $\mathscr{S}\subset\Omega_T$ with $|\mathscr{S}|<\delta_2$, it follows that 
$$
\|\widehat u_{h,k}\|_{L^1(\mathscr{S})}<\varepsilon_3\quad\mbox{ for all }\quad h,k>0.
$$
As a result, if $M_2$ is taken sufficiently large so that  $|\mathscr{B}|<\delta_2$, we estimate the first term on the right-hand side as follows:
\begin{equation}\label{lm16:R3->0}
\|R_3\|_{L^1(\Omega_{T^*})}\le 2 c_{\Phi_\lambda} \varepsilon_3
+  C_{\rm err}M_2 c_{\Phi_\lambda''} h^2 \|\nabla\widehat v_{h,k}\|^2_{L^2(\Omega_{T})}.
\end{equation}

$\bullet$ For $R_4$, we follow the same approach as with $R_1$,  obtaining 
\begin{equation}\label{lm16:R4->0}
\|R_4\|_{L^2(\Omega_{T^*})}\le \varepsilon_1 T \|u_{0h}\|_{L^1(\Omega)}. 
\end{equation}

Finally, compiling \eqref{lm16:R1->0}-\eqref{lm16:R4->0} into \eqref{lm16:decomposition} yields, on noting \eqref{bnd:v^n_h}, that 
$$
\begin{array}{rcl}
\|\mathcal{I}_h[\Phi(\widehat v_{h,k})u_{u,k}]- \Phi(\widehat v_{h,k})u_{h,k}\|_{L^1(\Omega_T^*)}&\le& 2 \varepsilon_1 T \|u_{0h}\|_{L^1(\Omega)}
\\
&&+C^2_{\rm inv} C_{\rm err} c_{\Phi_\lambda'}  \|v_{0h}\|_{L^\infty(\Omega)} \varepsilon_2
\\
&&+ C_{\rm inv}C_{\rm err} c_{\Phi'_\lambda} \|u_{0h}\|_{L^1(\Omega)} T   M_1  h
\\
&&+2 c_{\Phi_\lambda} \varepsilon_3
\\
&&+  C_{\rm err}M_2 c_{\Phi_\lambda''} C_2 h^2.
\end{array}
$$
Let $\varepsilon>0$. Take $\varepsilon_1=\frac{\varepsilon}{8 T \|u_{0h}\|_{L^1(\Omega)}}$, $\varepsilon_2=\frac{\varepsilon}{4 C^2_{\rm inv} C_{\rm err} c_{\Phi_\lambda'}  \|v_{0h}\|_{L^\infty(\Omega)}}$  and $\varepsilon_3=\frac{\varepsilon}{8c_{\Phi}}$. Further, let $h$ be such that 
$$
C_{\rm inv}C_{\rm err} c_{\Phi'_\lambda} \|u_{0h}\|_{L^1(\Omega)} T   M_1  h+C_{\rm err}M_2 c_{\Phi_\lambda''} h^2 C_2\le \frac{\varepsilon}{4}.
$$
These adjustments ensure that, for each $\varepsilon>0$, there exists $h_0>0$ such that, for all $(h,k)$  with $h<h_0$, the following holds:  
$$
\|\mathcal{I}_h[\Phi(\widetilde v_{h,k})\widehat u_{u,k}]- \Phi(\widetilde v_{h,k})\widehat u_{h,k}\|_{L^1(\Omega_{T^*})}<\varepsilon, 
$$
which lead to 
$$
\mathcal{I}_h[\Phi(\widetilde v_{h,k})\widehat u_{h,k}]- \Phi(\widetilde v_{h,k})\widehat u_{h,k}(\x,t)\to 0\quad\mbox{ a. e. on } \Omega_T, 
$$
since for any $(\x,t)\in\Omega_T$, there exists $k>0$ sufficiently small such that $(\x,t)\in\Omega_{T^*}$. Then the proof of \eqref{lm16:lab1} results from applying the dominated convergence theorem, thereby resulting in \eqref{w_conv_L1L1:Ph(v_hk_hat)u_hk->Phi(v)u}. 
\end{proof}
\begin{remark} We introduced the auxiliary set $\Omega_{T^*}$ in order to avoid the assumption $u_0\in H^1(\Omega)$ or using a CFL condition.
\end{remark}

\section{Passage to the limit as $h,k\to 0$ }\label{sec:limit}
To pass to the limit as $h, k\to 0$, assume that $\varphi\in W^{1,1}(0,T; W^{2,\infty}_N(\Omega))\subset C^0([0,T]; W^{2,\infty}_N(\Omega))$, with $\varphi(T)=0$, where $W^{2,\infty}(\Omega)=\{\varphi\in W^{2,\infty}(\Omega)\,:\, \partial_\n\varphi=0 \mbox{ on } \partial\Omega \}$ and define $\widetilde\mathcal{P}_h\varphi\in W^{1,1}(0,T; X_h^{\int=0})$ such that, for each $t\in[0,T]$, 
$$
(\nabla\widetilde \mathcal{P}_h\varphi(t), \nabla \chi_h)=(\nabla \varphi(t), \nabla \chi_h)\quad\mbox{ for all }\quad\chi_h\in X_h^{\int=0}.
$$
It is straightforward to verify that the projection operator $\widetilde \mathcal{P}_h$ satisfies the stability estimate
$$
\|\nabla\widetilde\mathcal{P}_h\varphi\|_{L^2(\Omega)}\le \|\nabla\varphi\|_{L^2(\Omega)},
$$
and, by Poincaré-Wirtinger inequality,
\begin{equation}\label{stab_H1:Ph}
\|\widetilde\mathcal{P}_h\varphi\|_{H^1(\Omega)}\le C_{\rm sta} \|\varphi\|_{H^1(\Omega)},
\end{equation}
where $C_{\rm sta}=\max\{1, C_{\rm PW}\}$, with $C_{\rm PW}>0$ is the Poincaré-Wirtinger constant. 

From the regularity for $\varphi$, we infer that 
\begin{equation}\label{conv_W11W1inf:P_hphi->phi}
\widetilde\mathcal{P}_h\varphi\to \varphi\quad\mbox{ in } W^{1,1}(0,T; H^1(\Omega)) 
\end{equation}
holds as $h\to 0$ from the interpolation error \cite{Scott_Zhang_1990}, for $p\in[1,\infty]$, 
\begin{equation}\label{err_W1p->W2p:SZh}
\|\varphi - \mathcal{SZ}_h \varphi \|_{W^{1,p}(\Omega)}\le C_{\rm err} h \|\varphi\|_{W^{2,p}(\Omega)}. 
\end{equation}
As a result, $\{\nabla \widetilde\mathcal{P}_h \partial_t\varphi \}_{h>0}$ is equi-integrable in $L^2(\Omega_T)$, that is, that is, for any $\widetilde \varepsilon > 0$, there exists $\delta > 0$ such that, for any measurable set $\mathscr{S} \subset \Omega_T$, if $|\mathscr{S}| < \delta$, then
$$
\|\nabla\mathcal{P}_h\partial_t\varphi\|_{L^2(\mathscr{S})} < \widetilde\varepsilon \quad \text{for all } h.
$$

Applying definitions \eqref{def:w_hk_tilde}–\eqref{def:w_hk} to \eqref{Alg}, after using $\bar v_h=\mathcal{P}_h\varphi$,  allows us to rewrite it as: 
$$
\int_0^T (\partial_t u_{h,k},  \widetilde\mathcal{P}_h \varphi_h )_h\, {\rm d}t +\int_0^T (\nabla\mathcal{I}_h[\Phi(\widetilde v_{h,k}) \widehat u_{h,k}], \nabla\widetilde\mathcal{P}_h(t)) \,{\rm d}t=0.
$$
If one defines $-\Delta_h: X_h\to X_h$ such that 
$$
-(\Delta_h \bar \chi_h, \chi_h)=(\nabla\bar\chi_h, \nabla\chi_h)\quad\mbox{ for all }\chi_h\in X_h,
$$
it follows that
$$
-(\Delta_h \widetilde\mathcal{P}_h\varphi, \bar \chi_h)=(\nabla\varphi, \nabla\chi_h)=-(\Delta\varphi, \chi_h).
$$
Thus, integration by parts shows
\begin{equation}\label{sec6:lab6}
- \int_0^T (u_{h,k},   \partial_t \widetilde\mathcal{P}_h \varphi)_h\, {\rm d}t -\int_0^T (\mathcal{I}_h[\Phi(\widetilde v_{h,k}) \widehat u_{h,k}], \Delta\varphi) \,{\rm d}t=(u_{0h}, \widetilde\mathcal{P}_h\varphi(0))_h.
\end{equation}
Next, we will take the limit in \eqref{sec6:lab5} as $h, k\to 0$. Write
$$
\begin{array}{rcl}
\displaystyle
\int_0^T (u_{h,k},  \partial_t \widetilde\mathcal{P}_h \varphi_h)_h\, {\rm d}t&=&\displaystyle
\int_0^T (u_{h,k}-\widehat u_{h,k},\partial_t \widetilde\mathcal{P}_h \varphi_h)_h\,{\rm d}t
\\
&&\displaystyle 
+\int_0^T (u_{h,k} \partial_t \widetilde\mathcal{P}_h \varphi_h -\mathcal{I}_h[u_{h,k} \partial_t \widetilde\mathcal{P}_h \varphi_h], 1)\,{\rm d}t
\\
&&\displaystyle
+\int_0^T (\widehat u_{h,k}, \partial_t\widetilde\mathcal{P}_h\varphi_h)\, {\rm d}t.
\end{array}
$$

We start by proving 
\begin{equation}\label{sec6:lab0}
\int_0^T (u_{h,k}-\widehat u_{h,k},\partial_t \widetilde\mathcal{P}_h \varphi_h)\,{\rm d}t\to 0\quad{ as }\quad h,k\to 0.
\end{equation}
Indeed, for \eqref{sec6:lab0} to be proved, we proceed as follows:   
$$
\begin{array}{rcl}
\displaystyle
\int_0^T (u_{h,k}-\widehat u_{h,k}, \partial_t\widetilde\mathcal{P}_h\varphi)_h {\rm d}t&=&\displaystyle
\sum_{n=0}^{N-1} \int_{t_n}^{t_{n+1}} \frac{t-t_n}{k} (u^{n+1}_h - u^n_h, \partial_t\widetilde\mathcal{P}_h\varphi)_h\, {\rm d}t 
\\
\mbox{from \eqref{eq:psi_h}} &=&\displaystyle
 \sum_{n=0}^{N-1} \int_{t_n}^{t_{n+1}} \frac{t-t_n}{k}\Big[(\nabla(\psi^{n+1}_h-\psi^n_h), \nabla\partial_t\widetilde\mathcal{P}_h\varphi)
 \\
 &&\qquad+ (\psi^{n+1}_h-\psi^n_h, \partial_t\widetilde\mathcal{P}_h\varphi)_h\Big]\,{\rm d}t
\\
&\le&\displaystyle
 \sum_{n=0}^{N-1} \int_{t_n}^{t_{n+1}} \|\delta_t\psi^{n+1}_h\|_{H^1_h(\Omega)} \|\partial_t\widetilde\mathcal{P}_h\varphi\|_{H^1_h(\Omega)}\,{\rm d}t
\\
\mbox{ by \eqref{equiv:L2_and_L2h}}&\le &\displaystyle
C_{\rm sta}k^{\frac{1}{2}} \left(\sum_{n=0}^{N-1} \|\delta_t\psi^{n+1}_h\|^2_{H^1_h(\Omega)}\right)^\frac{1}{2} \|\partial_t\widetilde\mathcal{P}_h\varphi\|_{L^2(0,T; H^1(\Omega))}
\\
\mbox{ by \eqref{bnd:Phi(v^n_h)u^(n+1)_h in L2L2h}}&\le &\displaystyle
C_{\rm sta} C_1^\frac{1}{2} \|\partial_t\widetilde\mathcal{P}_h\varphi|\|_{L^2(0,T; H^1(\Omega))} k^{\frac{1}{2}}\to 0
\end{array}
$$
as $h,k\to 0$. 

Let 
$$
\mathscr{C}=\{(t,\sigma)\in[0,T]\times \Sigma_h\,:\, |\nabla\widetilde\mathcal{P}_h\partial_t\varphi|_\sigma| >\lambda\}.
$$
Then, by Chevyshev's inequality, we have, for $\lambda>0$, that
$$
|\mathscr{C}|\le \frac{\|\nabla\widetilde\mathcal{P}_h\partial_t\varphi\|^2_{L^2(\Omega)}}{\lambda^2}\le \frac{\|\nabla\partial_t\varphi\|^2_{L^2(\Omega)}}{\lambda^2}.
$$
Let $\lambda >0 $ be sufficiently large such that $ |\mathscr{C}|<\delta$, so that
\begin{equation}\label{sec6:lab1}
 \|\nabla\widetilde\mathcal{P}_h\partial_t\varphi\|_{L^2(\mathscr{C})}<\widetilde\varepsilon.
\end{equation}
Thus, on noting \eqref{com_err_L1->L1-W1inf:I_h}, it follows that
$$
\begin{array}{rcl}
\displaystyle
\int_0^T (u_{h,k} \partial_t \widetilde\mathcal{P}_h \varphi_h(t) -\mathcal{I}_h[u_{h,k} \partial_t \widetilde\mathcal{P}_h \varphi_h], 1)\,{\rm d}t&=&\displaystyle \int_0 ^T \sum_{\sigma\in\Sigma_h} \|u_{h,k}\partial_t \widetilde\mathcal{P}_h \varphi_h -\mathcal{I}_h[u_{h,k} \partial_t \widetilde\mathcal{P}_h \varphi_h]\|_{L^1(\sigma)}\, {\rm d}t
\\
 &\le&\displaystyle
C_{\rm err}  h \int_0^T \sum_{\sigma\in\Sigma_h} \|u_{h,k}\|_{L^1(\sigma)}  \|\nabla\widetilde\mathcal{P}_h \partial_t\varphi\|_{L^\infty(\sigma)}  {\rm d}t 
\\
&\le&\displaystyle
C_{\rm err}  h \int_0^T \sum_{\sigma\in\mathscr{C}(t)} \|u_{h,k}\|_{L^1(\sigma)}  \|\nabla\widetilde\mathcal{P}_h \partial_t\varphi\|_{L^\infty(\sigma)}  {\rm d}t 
\\
&&+\displaystyle
C_{\rm err}  h \int_0^T \sum_{\sigma\in\mathscr{C}^c(t)} \|u_{h,k}\|_{L^1(\sigma)}  \|\nabla\widetilde\mathcal{P}_h \partial_t\varphi\|_{L^\infty(\sigma)}  {\rm d}t 
\\
&:=& I_1+I_2,
\end{array}
$$
where
$$
\mathscr{C}(t)=\{\sigma\in\Sigma_h : (t,\sigma)\in\mathscr{C}\},
$$
with $\mathscr{C}^c(t)$ being its complementary. Then, we bound:
\begin{equation}\label{sec6:lab2}
\begin{array}{rcl}
I_1 &\le&\displaystyle
C_{\rm err}  \int_0^T \left(\sum_{\sigma\in\mathscr{C}(t)} \|u_{h,k}\|^2_{L^1(\sigma)} \right)^\frac{1}{2} \left(\sum_{\sigma\in\mathscr{C}(t)} h^2\|\nabla\widetilde\mathcal{P}_h \partial_t\varphi\|^2_{L^\infty(\sigma)}\right)^\frac{1}{2}  {\rm d}t 
\\
\mbox{by \eqref{mass_conv:u^n_h}}&\le&\displaystyle
C_{\rm err} \|u_{0h}\|_{L^1(\Omega)}  \int_0^T\left(\sum_{\sigma\in\Sigma_h} h^2\|\nabla\widetilde\mathcal{P}_h \partial_t\varphi\|^2_{L^\infty(\sigma)}\right)^\frac{1}{2}  {\rm d}t 
\\
&\le&\displaystyle
C_{\rm err} \|u_{0h}\|_{L^1(\Omega)}  \int_0^T\sum_{\sigma\in\Sigma_h} h^2\|\nabla\widetilde\mathcal{P}_h \partial_t\varphi\|^2_{L^\infty(\sigma)}  {\rm d}t 
\\
&=&\displaystyle
C_{\rm err} \|u_{0h}\|_{L^1(\Omega)} T^{\frac{1}{2}}  \left(\int_0^T \sum_{\sigma\in\mathscr{C}(t)} h^2 \|\nabla\widetilde\mathcal{P}_h \partial_t\varphi\|^2_{L^\infty(\sigma)}\right)^\frac{1}{2}
\\
\mbox{ by \eqref{inv_ineq_Wnp->Wmq}} &\le&\displaystyle
C_{\rm err} C_{\rm inv } \|u_{0h}\|_{L^1(\Omega)} T^{\frac{1}{2}} \|\nabla\widetilde\mathcal{P}_h \partial_t\varphi\|_{L^2(\mathscr{C})}
\\
\mbox{by \eqref{sec6:lab1}}&\le&C_{\rm err} C_{\rm inv} \|u_{0h}\|_{L^1(\Omega)} T^{\frac{1}{2}} \widetilde\varepsilon
\end{array}
\end{equation}
and
\begin{equation}\label{sec6:lab3}
\begin{array}{rcl}
I_2&= &\displaystyle
C_{\rm err} h  \int_0^T \sum_{\sigma\in\mathscr{C}^c(t)} \|u_{h,k}\|_{L^1(\sigma)}   \|\nabla\widetilde\mathcal{P}_h \partial_t\varphi\|_{L^\infty(\sigma)}
\\
&\le&\displaystyle
C_{\rm err} h \lambda T  \|u_{0h}\|_{L^1(\Omega)}.
\end{array}
\end{equation}

For any $ \varepsilon > 0 $, if 
$$\tilde\varepsilon=\frac{\varepsilon}{2 C_{\rm err} C_{\rm inv} \|u_{0h}\|_{L^1(\Omega)} T^{\frac{1}{2}}}$$ 
in \eqref{sec6:lab2} and
$ h$ is chosen sufficiently small in \eqref{sec6:lab3}  such that
$$
C_{\rm err} h \lambda \|u_{0h}\|_{L^1(\Omega)} \le \frac{\varepsilon}{2},
$$
then it follows that
$$
\left|\int_0^T (u_{h,k} \partial_t \widetilde\mathcal{P}_h \varphi_h(t) -\mathcal{I}_h[u_{h,k} \partial_t \widetilde\mathcal{P}_h \varphi_h], 1)\,{\rm d}t\right| \le \varepsilon.
$$
Hence, we conclude that $$ \int_0^T (u_{h,k} \partial_t \widetilde\mathcal{P}_h \varphi_h(t) -\mathcal{I}_h[u_{h,k} \partial_t \widetilde\mathcal{P}_h \varphi_h], 1)\,{\rm d}t \to 0\quad\mbox{as}\quad  h,k\to 0.$$

Because of \eqref{conv_W11W1inf:P_hphi->phi} and \eqref{w_conv_L1L1:u_hk->u}, it leads to
$$
\int_0^T (\widehat u_{h,k}, \partial_t\widetilde\mathcal{P}_h\varphi_h)\, {\rm d}t\to \int_0^T (u, \partial_t\varphi)\, {\rm d}t,
$$
since 
\begin{equation}\label{sec6:lab4}
\begin{array}{rcl}
\|\partial_t\varphi-\widetilde\mathcal{P}_h\partial_t\varphi\|_{L^\infty(\Omega)}&\le& \|\partial_t\varphi-\mathcal{SZ}_h\partial_t\varphi\|_{L^\infty(\Omega)}+\|\mathcal{SZ}_h\partial_t\varphi-\widetilde\mathcal{P}_h\partial_t\varphi\|_{L^\infty(\Omega)} 
\\
\mbox{from \eqref{err_W1p->W2p:SZh} for $p=\infty$ and \eqref{inv_ineq_Wnp->Wmq} }&\le& C_{\rm err }  h \|\partial_t\varphi\|_{W^{1,\infty}(\Omega)} + C_{\rm inv} h^{-\frac{2}{p}} \|\mathcal{SZ}_h\partial_t\varphi-\widetilde\mathcal{P}_h\varphi\|_{L^p(\Omega)}   
\\
&\le& C_{\rm err }  h \|\partial_t\varphi\|_{W^{1,\infty}(\Omega)} + C_{\rm inv} h^{-\frac{2}{p}} \|\mathcal{SZ}_h\partial_t\varphi-\widetilde\mathcal{P}_h\partial_t\varphi\|_{H^1(\Omega)}
\\
&\le&
2 C_{\rm err }  h \|\varphi\|_{W^{1,\infty}(\Omega)} + C_{\rm inv} h^{-\frac{2}{p}} \|\partial_t\varphi-\widetilde\mathcal{P}_h\partial_t\varphi\|_{H^1(\Omega)}
\\
\mbox{ from \eqref{err_W1p->W2p:SZh} for $p>2$}&\le&
2 C_{\rm err }  h \|\partial_t\varphi\|_{W^{1,\infty}(\Omega)} + C_{\rm inv} h^{1-\frac{2}{p}} \|\partial_t\varphi\|_{H^2(\Omega)},
\end{array}
\end{equation}
whereupon, as $h,k\to 0$,   
$$
\int_0^T (u_{h,k},  \partial_t \widetilde\mathcal{P}_h \varphi_h(t))_h\, {\rm d}t\to \int_0^T (u, \partial_t\varphi)\, {\rm d}t. 
$$

Invoking \eqref{conv_W11W1inf:P_hphi->phi} along with \eqref{w_conv_L1L1:Ph(v_hk_hat)u_hk->Phi(v)u}, we find  that
$$
-\int_0^T (\mathcal{I}_h[\Phi(\widetilde v_{h,k}) \widehat u_{h,k}], \Delta\varphi ) \,{\rm d}t
\to -\int_0^T (\Phi(v)u, \Delta\varphi) \,{\rm d}t\quad{ as }\quad h,k\to 0.
$$ 
The convergence of the initial term is guaranteed by \eqref{def:Q_h} and \eqref{sec6:lab4}:
$$
(u_{0h}, \widetilde\mathcal{P}_h\varphi(0))_h=(u_0, \widetilde\mathcal{P}_h\varphi(0))\to (u_0, \varphi(0))\quad\mbox{ as } h\to 0. 
$$

Let now $\varphi\in W^{1,1}(0,T; H^1(\Omega))$ with $\varphi(T)=0$. Analogously, we have, on making use of \eqref{def:w_hk} and \eqref{def:w_hk_hat} that equation $\eqref{Alg}_2$ can be recast,  after performing  integration by part in time,  as  
\begin{equation}\label{sec6:lab5}
-\int_0^T (v_{h,k},  \widetilde\mathcal{P}_h\partial_t\varphi)_h\, {\rm d}t +\int_0^T (\nabla \widehat v_{h,k}, \nabla \widetilde\mathcal{P}_h\varphi)_h \,{\rm d}t+\int_0^T (\widehat u_{h,k} \widehat v_{h,k}, \widetilde\mathcal{P}_h\varphi)_h\, {\rm d}t=(u_{0h}, \widetilde\mathcal{P}_h\varphi(0)). 
\end{equation}
It is readily seen that the convergence of \eqref{sec6:lab5} as $h,k\to 0$ follows pretty much the same arguments used for \eqref{sec6:lab6}.

This latter completes the proof of Theorem \ref{th:main}.
 
\bibliographystyle{alpha}
\bibliography{sample}
\end{document}